\documentclass[graybox]{svmult}

\usepackage{mathptmx}       % selects Times Roman as basic font
\usepackage{helvet}         % selects Helvetica as sans-serif font
\usepackage{courier}        % selects Courier as typewriter font
\usepackage{type1cm}        % activate if the above 3 fonts are
\usepackage{makeidx}         % allows index generation
\usepackage{graphicx}        % standard LaTeX graphics tool
\usepackage{multicol}        % used for the two-column index
\usepackage[bottom]{footmisc}% places footnotes at page bottom
\usepackage{amssymb,amsmath}

\usepackage{hyperref}
\usepackage{url}
\usepackage{color}
\usepackage{breakurl}

\makeindex             % used for the subject index
\numberwithin{equation}{section}

\begin{document}

\title*{From Quantum Systems to $L$-Functions: Pair Correlation Statistics and Beyond}
\author{Owen Barrett, Frank W. K. Firk, Steven J. Miller, and Caroline Turnage-Butterbaugh}
\institute{Owen Barrett, \at Yale University, New Haven, CT 06520, \email{owen.barrett@yale.edu}
\and Frank W. K. Firk, \at Yale University, New Haven, CT 06520 \email{fwkfirk@aol.com}
\and Steven J. Miller, \at Williams College, Williamstown, MA 01267 \email{Steven.Miller.MC.96@aya.yale.edu, sjm1@williams.edu}
\and Caroline Turnage-Butterbaugh, Duke University, Durham, NC 27708 \email{cturnagebutterbaugh@gmail.com}}

\maketitle

%\abstract*{Each chapter should be preceded by an abstract (10--15 lines long) that summarizes the content. The abstract will appear \textit{online} at \url{www.SpringerLink.com} and be available with unrestricted access. This allows unregistered users to read the abstract as a teaser for the complete chapter. As a general rule the abstracts will not appear in the printed version of your book unless it is the style of your particular book or that of the series to which your book belongs. Please use the 'starred' version of the new Springer \texttt{abstract} command for typesetting the text of the online abstracts (cf. source file of this chapter template \texttt{abstract}) and include them with the source files of your manuscript. Use the plain \texttt{abstract} command if the abstract is also to appear in the printed version of the book.}

\abstract{The discovery of connections between the distribution of energy levels of heavy nuclei and spacings between prime numbers has been one of the most surprising and fruitful observations in the twentieth century. The connection between the two areas was first observed through Montgomery's work on the pair correlation of zeros of the Riemann zeta function. As its generalizations and consequences have motivated much of the following work, and to this day remains one of the most important outstanding conjectures in the field, it occupies a central role in our discussion below. We describe some of the many techniques and results from the past sixty years, especially the important roles played by numerical and experimental investigations, that led to the discovery of the connections and progress towards understanding the behaviors. In our survey of these two areas, we describe the common mathematics that explains the remarkable universality. We conclude with some thoughts on what might lie ahead in the pair correlation of zeros of the zeta function, and other similar quantities.}

\tableofcontents

%%%%%%%%%%%%%%%%%%%%%%%%%%%%%%%%%%%%%%%%%%%%%%%%%%%%%%%%%%%%%%%%%%%%%%%%%%%%%%%%%%%%%%%%%%%%%%%%%%%%%%%%%%%%%%%%%%%%%%%%%%%%%%%%%%%%
%%%%%%%%%%%%%%%%%%%%%%%%%%%%%%%%%%%%%%%%%%%%%%%%%%%%%%%%%%%%%%%%%%%%%%%%%%%%%%%%%%%%%%%%%%%%%%%%%%%%%%%%%%%%%%%%%%%%%%%%%%%%%%%%%%%%
%%%%%%%%%%%%%%%%%%%%%%%%%%%%%%%%%%%%%%%%%%%%%%%%%%%%%%%%%%%%%%%%%%%%%%%%%%%%%%%%%%%%%%%%%%%%%%%%%%%%%%%%%%%%%%%%%%%%%%%%%%%%%%%%%%%%
%%%%%%%%%%%%%%%%%%%%%%%%%%%%%%%%%%%%%%%%%%%%%%%%%%%%%%%%%%%%%%%%%%%%%%%%%%%%%%%%%%%%%%%%%%%%%%%%%%%%%%%%%%%%%%%%%%%%%%%%%%%%%%%%%%%%

\section{Introduction}\label{sec:introduction}
\numberwithin{equation}{section}

%describe objects. start with newtonian, go to quantum, what are the issues? approximations?
%move to number theory.
%introduce the characters
%start by saying this article is about waiting, how long till something happens
%one of the most important questions ever!
%start with wishart in 1928
%move to wigner and motivation
%talk about importance of numerical experiments on resonances

%Physics and Mathematics have developed and thrived together for millennia, where questions in one often drive progress in the other. While at first we might expect an asymmetry in the relationship, as after all it is hard to imagine doing serious physics without the language of mathematics but there is much mathematics that seems far afield to physics, if such an asymmetry exists it is definitely not with number theory!

%We explore below some mutually beneficial connections between number theory and physics. As this is a vast topic, we have chosen to focus on Montgomery's pair correlation conjecture and related statistics. Similar behavior is seen wildly different systems, from zeros of $L$-functions to energy levels of heavy nuclei.

Montgomery's pair correlation conjecture posits that zeros of $L$-functions behave similarly to energy levels of heavy nuclei. The bridge between these fields is random matrix theory, a beautiful subject which has successfully modeled a large variety of diverse phenomena (see \cite{BBDS, KrSe} for a great example of how varied the systems can be). It is impossible in a short chapter to cover all the topics and connections; fortunately there is no need as there is an extensive literature. Our goal is therefore to briefly describe the history of the subject and the correspondences, concentrating on some of the main objects of interest and past successes, ending with a brief tour through a \emph{subset} of current work and a discussion of some of the open questions in mathematics. We are deliberately brief in areas that are well known or are extensively covered in the literature, and instead dwell at greater lengths on the inspiration from and interpretation through physics (see for example \S\ref{sec:lessonsnuclearphys}), as these parts of the story are not as well known but deserve to be (both for historical reasons as well as the guidance they can offer).

To this end, we begin with a short introduction to random matrix theory and a quick description of the main characters studied in this chapter. We then continue in \S\ref{sec:birthrmt} with a detailed exposition of the historical development of random matrix theory in nuclear physics in the 1950s and 1960s. We note the pivotal role played by the nuclear physics experimentalists in gathering data to support the theoretical conjectures; we will see analogues of these when we get to the work in the 1970s and 1980s on zeros of $L$-functions in \S\ref{sec:earlyyearsNTandRMT}. One of our main purposes is in fact to highlight the power of experimental data, be it data from a lab or a computer calculation, and show how attempts to explain such results influence the development and direction of subjects. We then shift emphasis to number theory in \S\ref{sec:classtopaircorr}, and see how studies on the class number problem led Montgomery to his famous pair correlation conjecture for the zeros of the Riemann zeta function. This and related statistics are the focus of the rest of the chapter; we describe what they are, what progress has been made (theoretically and numerically), and then turn to some open questions. Most of these open questions involve how the arithmetic of $L$-functions influences the behavior; remarkably the main terms in a variety of problems are independent of the finer properties of $L$-functions, and it is only in lower order terms (or, equivalently, in the rates of convergence to the random matrix theory behavior) that the dependencies on these properties surface. We then conclude in \S\ref{sec:futuretrendsquestions} with current questions and some future trends.

\bigskip\noindent
\textbf{Acknowledgements.} The third named author was partially supported by NSF grant DMS1265673. We thank our colleagues and collaborators over the years for many helpful discussions on these and related topics. One of us (Miller) was fortunate to be a graduate student at Princeton, and had numerous opportunities then to converse with John Nash on a variety of mathematical topics. It was always a joy sitting next to him at seminars. We are grateful for his kind invitation to contribute to this work, and his comments on an earlier draft. We would also like to thank Thomas Preu for a careful reading in 2023, which led to several typos being found and corrected.

%%%%%%%%%%%%%%%%%%%%%%%%%%%%%%%%%%%%%%%%%%%%%%%%%%%%%%%%%
%%%%%%%%%%%%%%%%%%%%%%%%%%%%%%%%%%%%%%%%%%%%%%%%%%%%%%%%%
%%%%%%%%%%%%%%%%%%%%%%%%%%%%%%%%%%%%%%%%%%%%%%%%%%%%%%%%%
\subsection{The Early Days: Statistics and Biometrics}\label{sec:earlydayswishart}

Though our main characters will be energy levels of nuclei and zeros of $L$-functions, the story of random matrix theory begins neither with physics nor with mathematics, but with statistics and biometrics. In 1928 John Wishart published an article titled \emph{The Generalised Product Moment Distribution in Samples from a Normal Multivariate} \cite{Wis} in Biometrika (see \cite{Wik} for a history of the journal, which we briefly recap). The journal was founded at the start of the century by Francis Galton, Karl Pearson, and Walter Weldon for the study of statistics related to biometrics. In the editors' introduction in the first issue (see also \cite{Wik}), they write: \begin{quote} It is intended that Biometrika shall serve as a means not only of collecting or publishing under one title biological data of a kind not systematically collected or published elsewhere in any other periodical, but also of spreading a knowledge of such statistical theory as may be requisite for their scientific treatment. \end{quote}

The question of interest for Wishart was that of estimating covariance matrices. The paper begins with a review of work to date on samples from univariate and bivariate populations, and issues with the determination of correlation and regression coefficients. After summarizing some of the work and formulas from Fisher, Wishart writes: \begin{quote} The distribution of the correlation coefficient was deduced by direct integration from this result. Further, K. Pearson and V. Romanovsky, starting from this fundamental formula, were able to deal with the regression coefficients. Pearson, in 1925, gave the mean value and standard deviation of the regression coefficient, while Romanovsky and Pearson, in the following year, published the actual distribution. \end{quote} After talking about the new problems that arise when dealing with three or more variates, he continues: \begin{quote} What is now asserted is that all such problems depend, in the first instance, on the determination of a fundamental frequency distribution, which will be a generalisation of equation (2). It will, in fact, be the simultaneous distribution in samples of the $n$ variances (squared standard deviations) and the $\frac{n(n-1)}{2}$ product moment coefficients. It is the purpose of the present paper to give this generalised distribution, and to calculate its moments up to the fourth order. The case of three variates will first be considered in detail, and thereafter a proof for the general $n$-fold system will be given. \end{quote}

In his honor the distribution of the sample covariance matrix (arising from a sample from a multivariate normal distribution) is called the Wishart distribution. More specifically, if we have an $n\times p$ matrix $X$ whose rows are independently drawn from a $p$-variate mean 0 normal distribution, the Wishart distribution is the density of the $p\times p$ matrices $X^T X$.

Several items are worth noting here. First, we have an ensemble (a collection) of matrices whose entries are drawn from a fixed distribution; in this case there are dependencies among the entries. Second, these matrices are used to model observable quantities of interest, in this case covariances. Finally, in his article he mentions an earlier work of his (published in the Memoirs of the Royal Meteorological Society, volume II, pages 29--37, 1928) which experimentally confirmed some of the results discussed, thus showing the connections between experiment and theory which play such a prominent role later in the story also played a key role in the founding.

It was not until almost thirty years later that random matrix theory, in the hands and mind of Wigner, bursts onto the physics scene, and then it will be almost another thirty years more before the connections with number theory emerge. Before describing these histories in detail, we end the introduction with a very quick tour of some of the quantities and objects we'll meet.

%%%%%%%%%%%%%%%%%%%%%%%%%%%%%%%%%%%%%%%%%%%%%%%%%%%%%%%%%
%%%%%%%%%%%%%%%%%%%%%%%%%%%%%%%%%%%%%%%%%%%%%%%%%%%%%%%%%
%%%%%%%%%%%%%%%%%%%%%%%%%%%%%%%%%%%%%%%%%%%%%%%%%%%%%%%%%
\subsection{Cast of Characters: Nuclei and $L$-functions}\label{sec:castofchars}

The two main objects we study are energy levels of heavy nuclei on the physics side, and zeros of the Riemann zeta function (or more generally $L$-functions) on the number theory side, especially Montgomery's pair correlation conjecture and related statistics. We give a full statement of the pair correlation conjecture, and results towards its proof, in \S\ref{sec:paircorrzeroszeta}. Briefly, given an ordered sequence of events (such as zeros on the critical line, eigenvalues of Hermitian matrices, energy levels of heavy nuclei) one can look at how often a difference is observed. The remarkable conjecture is that these very different systems exhibit similar behavior.

We begin with a review of some facts about these areas, from theories for their behavior to how experimental observations were obtained which shed light on the structures, and then finish the introduction with some hints at the similarities between these two very different systems. Parts of that section, as well as much of \S\ref{sec:birthrmt}, are expanded with permission from the survey article \cite{FM} written by two of the authors of this chapter for the inaugural issue of the open access journal Symmetry. The goal of that article was similar to this chapter, though there the main quantity discussed was Wigner's semi-circle law and not pair correlation.

Many, if not all, of the other survey articles in the subject concentrate on the mathematics and ignore the experimental physics. When writing the survey \cite{FM} the authors deliberately sought a balance, with the intention of sharing and elaborating on that vantage again in a later work to give a wider audience a more complete description of the development of the subjects, as other approaches are already available in the literature. We especially recommend to the reader Goldston's excellent survey article \emph{Notes on pair correlation of zeros and prime numbers} (see \cite{Go}) for an extended, detailed technical discussion; the purpose of this chapter is to complement this and other surveys by highlighting other aspects of the story, especially how Montgomery's work on the pair correlation of zeros of $\zeta(s)$ connects, through random matrix theory, a central object of study in number theory to our understanding of the physics of heavy nuclei.

%%%%%%%%%%%%%%%%%%%%%%%%%%%%
%%%%%%%%%%%%%%%%%%%%%%%%%%%%
%%%%%%%%%%%%%%%%%%%%%%%%%%%%
\subsubsection{Atomic Theory and Nuclei}\label{sec:atomictheoryandnuclei}

%For thousands of years philosophers and scientists have attempted to find a theory to describe the composition of matter. Though elements of our current atomic theory can be found in the writings of the Roman poet and philosopher Lucretius almost two thousand years ago, those early musings lacked experimental justification. The intervening centuries between then and now saw many other theories, such as J. J. Thomson's plum pudding model and Niels Bohr's planetary model, rise and fall before our current quantum perspective. As our purpose is not to describe the development of all these theories, we refer the reader to \fix{(ADD \fix{REFERENCES} FOR DEVELOPMENT OF NUCLEAR THEORIES FROM PAST TO PRESENT)}, and we instead concentrate on the aspects of the story most pertinent to the connections with number theory: the role experimental data played.

Experiments and experimental data played a crucial role in our evolving understanding of the atom. For example, Ernest Rutherford's gold foil experiment (performed by Hans Geiger and Ernest Marsden) near the start of the twentieth century demonstrated that J. J. Thomson's plum pudding model of an atom with negatively charged electrons embedded in a positively charged region was false, and that the atom had a very small positively charged nucleus with the electrons far away. These experiments involved shooting alpha particles at thin gold foils. Alpha particles are helium atoms without the electrons and are thus positively charged. While this positive charge was responsible for disproving the plum pudding model, such particles could not deeply probe the positively charged nucleus due to the strong repulsion from like charges. To make further progress into the structure of the atom in general, and the nucleus in particular,  another object was needed. A great candidate was the neutron (discovered by Chadwick in 1932); as it did not have a net charge, the electric force would play an immensely smaller role in its interaction with the nucleus than it did with the alpha particles.

The earliest studies of neutron induced reactions showed that the total neutron cross section\footnote{ A total neutron cross section is defined as $$\frac{\rm Number\ of\ events\ of\ all\ types\ per\ unit\ time\ per\ nucleus}{\rm Number\ of\ incident\ neutrons\ per\ unit\ time\ per\ unit\ area}, $$ and has the dimensions of area (the standard unit is the \emph{barn}, $10^{-24} {\rm cm}^2$).} for the interaction of low-energy (electron-volt, eV) neutrons with a nucleus is frequently much greater than the geometrical area presented by the target nucleus to the incident neutron \cite{FA}.  It was also found that the cross section varies rapidly as a function of the bombarding energy of the incident neutron.  The appearance of these well-defined \emph{resonances} in the neutron cross section is the most characteristic feature of low energy nuclear reactions. %%%%%%%%%%%%%%% (see Figures 1a, b). \fix{ADD FIGURE}

In general, the low energy resonances were found to be closely spaced (spacing $\le 10$ eV in heavy nuclei), and to be very narrow (widths $\le 0.1$ eV).  These facts led Niels Bohr to introduce the \emph{compound nucleus} model \cite{BO} that assumes the interaction between an incoming neutron and the target nucleus is so strong that the neutron rapidly shares its energy with many of the target nucleons.  The nuclear state that results from the combination of incident neutron and target nucleus may therefore last until sufficient energy again resides in one of the nucleons for it to escape from the system.  This is a statistical process, and a considerable time may elapse before it occurs.  The long lifetime of the state ($\tau$) (on a nuclear timescale) explains the narrow width ($\Gamma$) of the resonance.\footnote{The width, $\Gamma$, is related to the lifetime, $\tau$, by the uncertainty relation $\Gamma = h/2\pi \tau$, where $h$  is  Planck's constant.  The finite width (lack of energy definition) is due to the fact that a resonant state can decay by emitting a particle, or radiation, whereas a state of definite energy  must be a stationary state.}  Also, since many nucleons are involved in the formation of a compound state, the close spacing of the resonances is to be expected since there are clearly many ways of exciting many nucleons.  The qualitative model outlined above has formed the basis of most theoretical descriptions of low-energy, resonant nuclear reactions \cite{BW}.

If a resonant state can decay in a number of different ways (or channels), we can ascribe a probability per unit time for the decay into a channel, $c$, which can be expressed as a partial width $\Gamma_{\lambda c}$.  The total width is the sum of the partial widths, i.e., $\Gamma_\lambda = \sum_c \Gamma_{\lambda c}$.

The appearance of well-defined resonances occurs in heavy nuclei (mass number $A \ge 100$, say) for incident neutron energies up to about $100$ keV, and in light nuclei up to neutron energies of several MeV.  As the neutron bombarding energies are increased above these energies, the total cross sections are observed to become smoother functions of neutron energy \cite{HS}.  This is due to two effects: firstly, the level density (i.e., the number of resonances per unit energy interval) increases rapidly as the excitation energy of the compound nucleus is increased, and secondly, the widths of the individual resonances tend to increase with increasing excitation energy so that, eventually, they overlap.  The smoothed-out cross sections provide useful information on the average properties of resonances.  One of the most significant features of these cross sections is the appearance of gross fluctuations that have been interpreted in terms of the single-particle nature of the neutron-nucleus interaction \cite{LTW}.  These \emph{giant resonances} form one of the main sources of experimental evidence for introducing the successful \emph{optical model} of nuclear reactions.  This model represents the interaction between a neutron and a nucleus in terms of the neutron moving in a complex potential well \cite{OBJ} in which the imaginary part allows for the absorption of the incident neutron.

Experimental results show that, on increasing the bombarding energy above about $5$ MeV, a different reaction mechanism may occur.  For example, the energy spectra of emitted nucleons frequently contain too many high-energy nucleons compared with the predictions of the compound nucleus model.  The mechanism no longer appears to be one in which the incident neutron shares its energy with many target nucleons but is one in which the neutron interacts with a single nucleon or, at most, a few nucleons.  Such a mechanism is termed a \emph{direct interaction}, which is defined as a nuclear reaction in which only a few of the available degrees of freedom of the system are involved \cite{AU}.

The \emph{optical model}, mentioned above, is an important example of a direct interaction that takes place even at low bombarding energies.  The incident neutron is considered to move in the mean nuclear potential of all the nucleons in the target.  This model also has been used to account for anomalies in the spectra of gamma-rays resulting from thermal neutron capture \cite{L,LL}.

At even higher bombarding energies, greater than $50$ MeV, say, the mechanism becomes clearer in the sense that direct processes are the most important.  The reactions then give information on the fundamental nucleon-nucleon interaction; these studies and their interpretation are, however, outside the scope of the present discussion.

When a low-energy neutron (energy $< 10$ keV, say) interacts with a nucleus the excitation energy of the compound nucleus is greatly increased by the neutron binding energy that typically ranges from $5$ to $10$ MeV.  In the late 1950s, experimental methods were developed for measuring low-energy neutrons with resolutions of a few electron-volts.  This meant that, for the first time in any physical system, it became possible to study the fine structure of resonances at energies far above the ground state of the system.  The relevant experimental methods are discussed in \S\ref{sec:birthrmt}.  Important information was thereby obtained concerning the properties that characterizes the resonances such as their peak cross sections, elastic scattering widths, and adjacent spacing.  The results were  used to test the predictions of various nuclear models used to describe the interactions.  These models ranged from the Fermi Gas Model, a quantized version of classical Statistical Mechanics and Thermodynamics \cite{BE}, to the sophisticated Nuclear Shell Model \cite{BW}.  In the mid-1950s, all Statistical Mechanics Models predicted that the spacing distribution of nearest-neighbor resonances of the same spin and parity in a heavy nucleus (mass number $A \ge 100$, say) was an exponential distribution.  By 1956, the experimental evidence on the spacing distribution of s-wave resonances in a number of heavy nuclei indicated a lack of very closely-spaced resonances, contradicting the predictions of an exponential distribution \cite{HH}.  By 1960, two research groups \cite{RDRH,FLM} showed, unequivocally, that the spacing distribution of resonances up to an energy of almost $2$ keV followed the prediction of the random matrix model surmised by Wigner in 1956 \cite{Wig5}; in his model the probability of a zero spacing is zero!  It is a model rooted in statistics, which interestingly is where our story on random matrix theory began!

%%%%%%%%%%%%%%%%%%%%%%%%%%%%
%%%%%%%%%%%%%%%%%%%%%%%%%%%%
%%%%%%%%%%%%%%%%%%%%%%%%%%%%
\subsubsection{$L$-functions and Their Zeros}\label{sec:Lfunc_and_zeros}

There are many excellent introductions, at a variety of levels, to number theory and $L$-functions. We assume the reader is familiar with the basics of the subject; for more details see among others \cite{Da,Ed,HW,IK,MT-B,Se}. The discussion below is a quick review and is an abridgement (and slight expansion) of \cite{FM}, which has additional details.

The primes are the building blocks of number theory: every integer can be written uniquely as a product of prime powers. Note that the role played by the primes mirrors that of atoms in building up molecules. One of the most important questions we can ask about primes is also one of the most basic: how many primes are there at most $x$? In other words, how many building blocks are there up to a given point?

Euclid proved over 2000 years ago that there are infinitely many primes; so, if we let $\pi(x)$ denote the number of primes at most $x$, we know $\lim_{x\to\infty} \pi(x) = \infty$. Though Euclid's proof is still used in courses around the world (and gives a growth rate on the order  of $\log\log x$), one can obtain much better counts on $\pi(x)$.

The prime number theorem states that the number of primes at most $x$ is ${\rm Li}(x) + o({\rm Li}(x))$, where ${\rm Li}(x) = \int_2^x dt/\log t$ and for $x$ large, ${\rm Li}(x)$ is approximately $x/\log x$, and $f(x) = o(g(x))$ means $\lim_{x\to\infty} f(x)/g(x) = 0$. While it is possible to prove the prime number theorem elementarily \cite{Erd,Sel2}, the most informative proofs use complex numbers and complex analysis, and lead to the fascinating connection between number theory and nuclear physics. One of the most fruitful approaches to understanding the primes is to understand properties of the Riemann zeta function, $\zeta(s)$, which is defined for ${\rm Re}(s) > 1$ by \begin{equation} \zeta(s) \ = \ \sum_{n=1}^\infty \frac1{n^s}; \end{equation} the series converges for ${\rm Re}(s) > 1$ by the integral test. By unique factorization, we may also write $\zeta(s)$ as a product over primes. To see this, use the geometric series formula to expand $(1-p^{-s})^{-1}$ as $\sum_{k=0}^\infty p^{-ks}$ and note that $n^{-s}$ occurs exactly once on each side (and clearly every term from expanding the product is of the form $n^{-s}$ for some $n$). This is called the Euler product of $\zeta(s)$, and is one of its most important properties: \begin{equation} \zeta(s) \ = \ \sum_{n=1}^\infty \frac1{n^s} \ = \ \prod_{p\ {\rm prime}} \left(1 - \frac1{p^s}\right)^{-1}. \end{equation} Initially defined only for ${\rm Re}(s) > 1$, using complex analysis the Riemann zeta function can be meromorphically continued to all of $\mathbb{C}$, having only a simple pole with residue 1 at $s=1$. It satisfies the functional equation \begin{equation}\label{eq:defixi} \xi(s) \ =\ \frac12 s(s-1) \Gamma\left(\frac{s}{2}\right)\pi^{-\frac{s}{2}} \zeta(s) \ = \ \xi(1-s).\end{equation} One proof is to use the Gamma function, $\Gamma(s) = \int_0^\infty e^{-t} t^{s-1}dt$.  A simple change of variables gives \begin{equation} \int_0^\infty x^{\frac12 s - 1} e^{-n^2 \pi x}dx \ = \ \Gamma\left(\frac{s}{2}\right) / n^s\pi^{s/2}.\end{equation} Summing over $n$ represents a multiple of $\zeta(s)$ as an integral. After some algebra we find \begin{equation} \Gamma\left(\frac{s}2\right) \zeta (s) \ = \ \int_1^\infty  x^{\frac12 s -1} \omega(x)dx + \int_1^\infty x^{-\frac12 s -1}\omega\left(\frac{1}{x}\right)dx,\end{equation} with $\omega(x)$ $=$ $\sum_{n=1}^{\infty} e^{-n^2 \pi x}$.  Using Poisson summation, we see \begin{equation} \omega\left(\frac{1}{x}\right)\ =\ -\frac12 + -\frac12 x^\frac12 + x^\frac12 \omega(x),\end{equation} which yields \begin{equation} \pi^{-\frac12 s} \Gamma\left(\frac{s}2\right) \zeta (s) \ = \ \frac{1}{s(s-1)} + \int_1^\infty  (x^{\frac12 s -1}+ x^{- \frac12 s - \frac12}) \omega(x)dx,\end{equation} from which the claimed functional equation follows.

The distribution of the primes is a difficult problem; however, the distribution of the positive integers is not and has been completely known for quite some time! The hope is that we can understand $\sum_n 1/n^s$ as this involves sums over the integers, and somehow pass this knowledge on to the primes through the Euler product.

Riemann \cite{Ri} (see \cite{Cl,Ed} for an English translation) observed a fascinating connection between the zeros of $\zeta(s)$ and the error term in the prime number theorem. As this relation is the starting point for our story on the number theory side, we describe the details in some length. One of the most natural things to do to a complex function is to take contour integrals of its logarithmic derivative; this yields information about zeros and poles, and  we will see later in \eqref{eq:genexpformula} that we can get even more information if we weigh the integral with a test function. There are two expressions for $\zeta(s)$; however, for the logarithmic derivative it is clear that we should use the Euler product over the sum expansion, as the logarithm of a product is the sum of the logarithms. Let  \begin{equation}\label{eq:defnlambda} \Lambda(n)\  = \ \begin{cases} \log p & {\rm if}\ n=p^r\ {\rm for\ some\ integer\ }r \\ 0 & {\rm otherwise.} \end{cases}\end{equation} We find \begin{equation}\label{eq:logderiv} \frac{\zeta'(s)}{\zeta(s)} \ = \ -\sum_p \frac{\log p \cdot p^{-s}}{1-p^{-s}} \ = \ -\sum_{n=1}^\infty \frac{\Lambda(n)}{n^s}  \end{equation} (this is proved by using the geometric series formula to write $(1-p^{-s})^{-1}$ as $\sum_{k=0}^\infty 1/p^s$, collecting terms and then using the definition of $\Lambda(n)$). Moving the negative sign over and multiplying by $x^s/s$, we find \begin{equation} \frac1{2\pi i} \int_{(c)} -\frac{\zeta'(s)}{\zeta(s)} \frac{x^s}{s}\ ds \ = \ \frac1{2\pi i} \int_{(c)} \sum_{n \le x} \Lambda(n)\left(\frac{x}{n}\right)^s \frac{ds}{s}, \end{equation} where we are integrating over some line ${\rm Re}(s) = c > 1$. The integral on the right hand side is $1$ if $n < x$ and $0$ if $n > x$ (by choosing $x$ non-integral, we do not need to worry about $x=n$), and thus gives $\sum_{n \le x} \Lambda(n)$. By shifting contours and keeping track of the poles and zeros of $\zeta(s)$, the residue theorem implies that the left hand side is \begin{equation} x - \sum_{\rho: \zeta(\rho) = 0} \frac{x^\rho}{\rho}; \end{equation} the $x$ term comes from the pole of $\zeta(s)$ at $s=1$ (remember we count poles with a minus sign), while the $x^\rho/\rho$ term arises from zeros; in both cases we must multiply by the residue, which is $x^\rho/\rho$ (it can be shown that $\zeta(s)$ has neither a zero nor a pole at $s=0$). Some care is required with this sum, as $\sum 1/|\rho|$ diverges. The solution involves pairing the contribution from $\rho$ with $\overline{\rho}$; see for example \cite{Da}.

The Riemann zeta function vanishes whenever $\rho$ is a negative even integer; we call these the \emph{trivial} zeros. These terms contribute $\sum_{k=-1}^\infty x^{-2k}/(2k) = -\frac12 \log(1-x^{-2})$. This leads to the following beautiful formula, known as the \emph{explicit formula}: \begin{equation}\label{eq:keybeautyformriemannprimes}  x - \sum_{\rho: {\rm Re}(\rho) \in (0,1)\atop \zeta(\rho)=0} \frac{x^\rho}{\rho} - \frac12\log(1-x^{-2}) \ = \ \sum_{n \le x} \Lambda(n)\end{equation} If we write $n$ as $p^r$, the contribution from all $p^r$ pieces with $r \ge 2$ is bounded by $2x^{1/2}\log x$ for $x$ large, thus we really have a formula for the sum of the primes at most $x$, with the prime $p$ weighted by $\log p$. Through partial summation, knowing the weighted sum is equivalent to knowing the unweighted sum.

We can now see the connection between the zeros of the Riemann zeta function and counting primes at most $x$. The contribution from the trivial zeros is well-understood, and is just $-\frac12\log(1-x^{-2})$. The remaining zeros, whose real parts are in $[0,1]$, are called the \emph{non-trivial} or \emph{critical} zeros. They are far more important and more mysterious. The smaller the real part of these zeros of $\zeta(s)$, the smaller the error. Due to the functional equation, however, if $\zeta(\rho) = 0$ for a critical zero $\rho$ then $\zeta(1-\rho) = 0$ as well. Thus the `smallest' the real part can be is 1/2. This is the celebrated \emph{Riemann Hypothesis (RH)}, which is probably the most important mathematical aside ever in a paper. Riemann \cite{Cl,Ed,Ri} wrote (translated into English; note when he talks about the roots being real, he's writing the roots as $1/2 + i\gamma$, and thus $\gamma \in \mathbb{R}$ is the Riemann Hypothesis):
\begin{quote} One now finds indeed approximately this number of real roots within these limits, and it is very probable that all roots are real. Certainly one would wish for a stricter proof here; I have meanwhile temporarily put aside the search for this after some fleeting futile attempts, as it appears unnecessary for the next objective of my investigation.
\end{quote}
Though not mentioned in the paper, Riemann had developed a terrific formula for computing the zeros of $\zeta(s)$, and had checked (but never reported!) that the first few were on the critical line ${\rm Re}(s) = 1/2$. His numerical computations were only discovered decades later when Siegel was looking through Riemann's papers.

RH has a plethora of applications throughout number theory and mathematics; counting primes is but one of many. The prime number theorem is in fact equivalent to the statement that ${\rm Re}(\rho) < 1$ for any zero of $\zeta(s)$, and was first proved independently by Hadamard \cite{Had} and de la Vall\'{e}e Poussin \cite{dlVP} in 1896. Each proof crucially used results from complex analysis, which is hardly surprising given that Riemann had shown $\pi(x)$ is related to the zeros of the meromorphic function $\zeta(s)$. It was not until almost 50 years later that Erd\"{o}s \cite{Erd} and Selberg \cite{Sel2} obtained elementary proofs of the prime number theorem (in other words, proofs that did not use complex analysis, which was quite surprising as the prime number theorem was known to be equivalent to a statement about zeros of a meromorphic function). See \cite{Gol4} for some commentary on the history of elementary proofs. It is clear, however, that the distribution of the zeros of the Riemann zeta function will be of primary (in both senses of the word!) importance.

The Riemann zeta function is the first of many similar functions that we can study. We assume the reader has seen $L$-functions before; in addition to the surveys mentioned earlier, see also the introductory remarks in \cite{ILS, RS}. We can examine, for the real part of $s$ sufficiently large, \begin{equation} L(s,f) \ := \ \sum_{n=1}^\infty \frac{a_f(n)}{n^s};\end{equation} of course, while we can create such a function for any sequence $\{a_f(n)\}$ of sufficient decay, only certain choices will lead to useful objects whose zeros encode the solution to questions of arithmetic interest. For example, if we chose $a_f$ arising from Dirichlet characters we obtain information about primes in arithmetic progression, while taking $a_f(p)$ to count the number of solutions to an elliptic curve $y^2 = x^3 + Ax + B$ modulo $p$ yields information about the rank of the group of rational solutions.

Our previous analysis, where many of our formulas are due to taking the logarithmic derivative and computing a contour integral, suggests that we insist that an Euler product hold: \begin{equation} L(s,f) \ = \ \sum_{n=1}^\infty \frac{a_f(n)}{n^s} \  =  \ \prod_{p\ {\rm prime}} L_p(s,f). \end{equation} Further, we want a functional equation relating the values of the completed $L$-function at $s$ and $1-s$, which allows us to take the series expansion that originally converges only for real part of $s$ large and obtain a function defined everywhere:
\begin{eqnarray}
\Lambda(s,f) \ = \ L_\infty(s,f) L(s,f) \ = \  \epsilon_f \Lambda(1-s,f),
\end{eqnarray} where $\epsilon_f$, the sign of the functional equation, is of absolute value 1, and
\begin{eqnarray} L_p(s,f) & \ = \ & \prod_{j=1}^d \left(1 - \alpha_{f;j}(p) p^{-s}\right)^{-1} \nonumber\\ L_\infty(s,f) & = & A Q^s \prod_{j=1}^n \Gamma\left(\frac{s}{2} + \alpha_{f;j}\right), \end{eqnarray} with $A\neq 0$ a complex number, $Q > 0$, $\alpha_{f;j} \ge 0$ and $\sum_{j=1}^n \alpha_{f;j}(p)^\nu = a_f(p^\nu)$. For `nice' $L$-functions, it is believed that the Generalized Riemann Hypothesis (GRH) holds: All non-trivial zeros real part equal to 1/2.

We end our introduction to our main number theoretic objects of interest by noting that \eqref{eq:keybeautyformriemannprimes} is capable of massive generalization, not just to other $L$-functions but we can multiply \eqref{eq:logderiv} by a nice test function $\phi(s)$ instead of the specific function $x^s/s$. The result of this choice is to have a formula that relates sums of $\phi$ at zeros of our $L$-function to sums of the Fourier transform of $\phi$ at the primes. For example (see Section 4 of \cite{ILS}) one can show \begin{equation}\label{eq:genexpformula} \sum_{\rho} \phi\left(\frac{\gamma}{2\pi}{\log R}\right) \ = \ \frac{A}{\log R} - 2 \sum_p \sum_{\nu=1}^\infty a_f(p^\nu) \widehat{\phi}\left(\frac{\log p^\nu}{\log R}\right)  \frac{\log p}{p^{\nu/2} \log R}, \end{equation} where $R$ is a free scaling parameter chosen for the problem of interest, $A = 2\widehat{\phi}(0) \log Q + \sum_{j=1}^n A_j$ with \begin{equation} A_j  \ = \ \int_{-\infty}^\infty \psi\left(\alpha_{f;j} + \frac14 + \frac{2\pi i x}{\log R}\right) \phi(x) dx \end{equation} and the Fourier transform is defined by \begin{equation} \widehat{\phi}(y) \ := \ \int_{-\infty}^\infty \phi(x) e^{-2\pi i x y} dx. \end{equation}

%%%%%%%%%%%%%%%%%%%%%%%%%%%%
%%%%%%%%%%%%%%%%%%%%%%%%%%%%
%%%%%%%%%%%%%%%%%%%%%%%%%%%%
\subsubsection{From the Hilbert-P\'olya Connection to Random Matrix Theory}

As stated earlier, the Generalized Riemann Hypothesis asserts that the non-trivial zeros of the an $L$-function are of the form $\rho = 1/2+i\gamma_\rho$ with $\gamma_\rho$ real. Thus it makes sense to talk about the distribution between adjacent zeros. Around 1913, P\'{o}lya conjectured that the $\gamma_\rho$ are the eigenvalues of a naturally occurring, unbounded, self-adjoint operator, and are therefore real.\footnote{If $v$ is an eigenvector with eigenvalue $\lambda$ of a Hermitian matrix $A$ (so $A = A^\ast$ with $A^\ast$ the complex conjugate transpose of $A$, then $v^\ast (Av) = v^\ast (A^\ast v) = (Av)^\ast v$; the first expression is $\lambda ||v||^2$ while the last is $\overline{\lambda} ||v||^2$, with $||v||^2 = v^\ast v = \sum |v_i|^2$ non-zero. Thus $\lambda = \overline{\lambda}$, and the eigenvalues are real. This is one of the most important properties of Hermitian matrices, as it allows us to order the eigenvalues.}\label{footnote:hermevaluesreal} Later, Hilbert contributed to the conjecture, and reportedly introduced the phrase `spectrum' to describe the eigenvalues of an equivalent Hermitian operator, apparently by analogy with the optical spectra observed in atoms.  This remarkable analogy pre-dated Heisenberg's Matrix Mechanics and the Hamiltonian formulation of Quantum Mechanics by more than a decade.

Not surprisingly, the Hilbert-P\'{o}lya conjecture was considered so intractable that it was not pursued for decades, and random matrix theory remained in a dormant state.  To quote Diaconis \cite{Di1}: \begin{quote} Historically, random matrix theory was started by statisticians \cite{Wis} studying the correlations between different features of population (height, weight, income...).  This led to correlation matrices with $(i, j)$ entry the correlation between the $i$th  and $j$th features.  If the data were based on a random sample from a larger population, these correlation matrices are random; the study of how the eigenvalues of such samples fluctuate was one of the first great accomplishments of random matrix theory. \end{quote}

Diaconis \cite{Di2} has given an extensive review of random matrix theory from the perspective of a statistician.  A strong argument can be made, however, that random matrix theory, as we know it today in the physical sciences, began in a formal mathematical sense with the Wigner surmise \cite{Wig5} concerning the spacing distribution of adjacent resonances (of the same spin and parity) in the interactions between low-energy neutrons and nuclei, which we describe in great detail in \S\ref{sec:birthrmt}.

%%%%%%%%%%%%%%%%%%%%%%%%%%%%%%%%%%%%%%%%%%%%%%%%%%%%%%%%%%%%%%%%%%%%%%%%%%%%%%%%%%%%%%%%%%%%%%%%%%%%%%%%%%%%%%%%%%%%%%%%%%%%%%%%%%%%
%%%%%%%%%%%%%%%%%%%%%%%%%%%%%%%%%%%%%%%%%%%%%%%%%%%%%%%%%%%%%%%%%%%%%%%%%%%%%%%%%%%%%%%%%%%%%%%%%%%%%%%%%%%%%%%%%%%%%%%%%%%%%%%%%%%%
%%%%%%%%%%%%%%%%%%%%%%%%%%%%%%%%%%%%%%%%%%%%%%%%%%%%%%%%%%%%%%%%%%%%%%%%%%%%%%%%%%%%%%%%%%%%%%%%%%%%%%%%%%%%%%%%%%%%%%%%%%%%%%%%%%%%
%%%%%%%%%%%%%%%%%%%%%%%%%%%%%%%%%%%%%%%%%%%%%%%%%%%%%%%%%%%%%%%%%%%%%%%%%%%%%%%%%%%%%%%%%%%%%%%%%%%%%%%%%%%%%%%%%%%%%%%%%%%%%%%%%%%%
\section{The `Birth' of Random Matrix Theory in Nuclear Physics}\label{sec:birthrmt}
\numberwithin{equation}{section}

Below we discuss some of the history of investigations of the nucleus, concentrating on the parts that led to the introduction of random matrix theory to the subject. As mentioned earlier, this section is expanded with permission from \cite{FM}. Our goal is to provide the reader with both sides of the coin, highlighting the interplay between theory and experiment, and building the basis for applications to understanding zeros of $L$-functions; we have chosen to spend a good amount of space on these experiments and conjectures as these are less-well known to the general mathematician than the later parts of our story.

While other methods have since been developed, random matrix theory was the first to make truly accurate, testable predictions. The general idea is that the behavior of zeros of $L$-functions are well-modeled by the behavior of eigenvalues of certain matrices. This idea had previously been successfully used to model the distribution of energy levels of heavy nuclei (some of the fundamental papers and books on the subject, ranging from experiments to theory, include \cite{BFFMPW, DLL, Dy1, Dy2, FLM, FRG, For, FKPT, Gau, HH, HPB, Hu, Meh1, Meh2, MG, MT-B, Po, T, Wig1, Wig2, Wig3, Wig4, Wig5, Wig6}). We describe the development of random matrix theory in nuclear physics below, and then delve into more of the details of the connection between the two subjects.

%%%%%%%%%%%%%%%%%%%%%%%%%%%%%%%%%%%%%%%%%%%%%%%%%%%%%%%%%
%%%%%%%%%%%%%%%%%%%%%%%%%%%%%%%%%%%%%%%%%%%%%%%%%%%%%%%%%
%%%%%%%%%%%%%%%%%%%%%%%%%%%%%%%%%%%%%%%%%%%%%%%%%%%%%%%%%
\subsection{Neutron Physics}

The period from the mid-1930s to the late 1970s was the golden age of neutron physics; widespread interest in understanding the physics of the nucleus, coupled with the need for accurate data in the design of nuclear reactors, made the field of neutron physics of global importance in fundamental physics, technology, economics, and politics. In \S\ref{sec:atomictheoryandnuclei} we introduced some of the early models for nuclei, and discussed some of the original experiments. In this section we describe later work where better resolution was possible. Later we will show how a similar perspective and chain of progress holds in studies of zeros of the Riemann zeta function! Thus the material here, in addition to being of interest in its own right, will also provide a valuable vantage for study of arithmetic objects.

In the mid-1950s, a discovery was made that turned out to have far-reaching consequences beyond anything that those working in the field could have imagined.  For the first time, it was possible to study the microstructure of the continuum in a strongly-coupled, many-body system, at very high excitation energies. This unique situation came about as the result of the following facts.

\begin{itemize}

\item	Neutrons, with kinetic energies of a few electron-volts, excite states in
compound nuclei at energies ranging from about 5 million electron-volts to almost 10 million electron-volts -- typical neutron binding energies. Schematically, see Figure \ref{fig:groundstatesneutrons}.

\begin{figure}
\begin{center}
\scalebox{.7}{\includegraphics{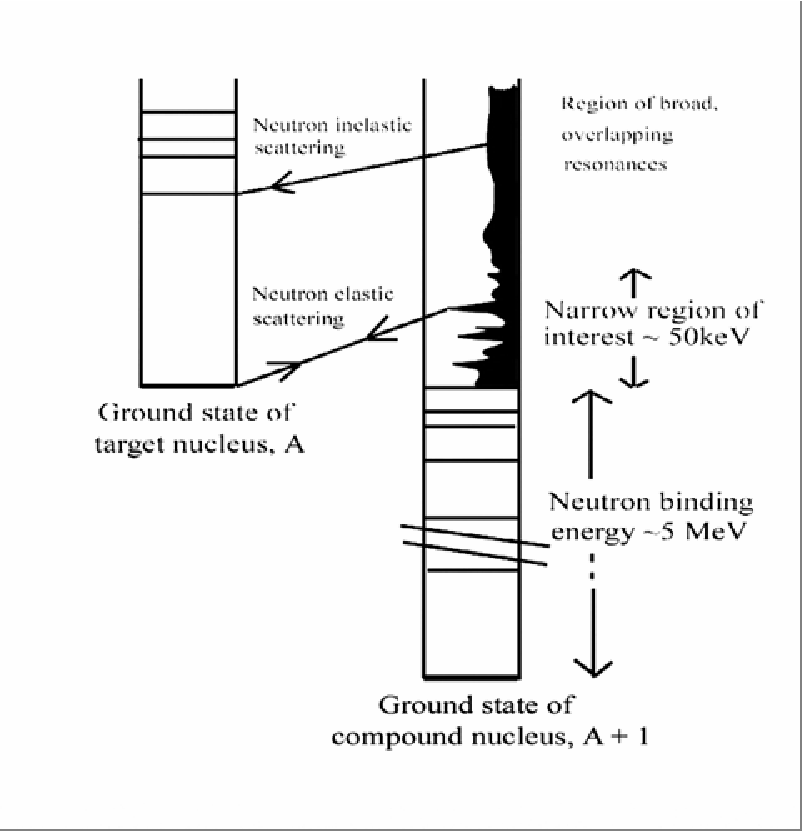}}
\caption{\label{fig:groundstatesneutrons} An energy-level diagram showing the location of  highly-excited resonances in the compound nucleus formed by the interaction of a neutron, $n$, with a nucleus of mass number $A$.  Nature provides us with a narrow energy region in which the resonances are clearly separated, and are observable.}
\end{center}
\end{figure}

\item	Low-energy resonant states in heavy nuclei (mass numbers greater than about $100$) have lifetimes in the range $10^{-14}$ to $10^{-15}$ seconds, and therefore they have widths of about $1$ eV.  The compound nucleus loses all memory of the way in which it is formed.  It takes a relatively long time for sufficient energy to reside in a neutron before being emitted.  This is a highly complex, statistical process.  In heavy nuclei, the average spacing of adjacent resonances is typically in the range from a few eV to several hundred eV.

\item	Just above the neutron binding energy, the angular momentum barrier restricts the possible range of values of total spin of a resonance, \textbf{J} (\textbf{J} = \textbf{I} + \textbf{i} + \textbf{l}, where \textbf{I} is the spin of the target nucleus, \textbf{i} is the neutron spin, and \textbf{l} is the relative orbital angular momentum).  This is an important technical point.	

\item	The neutron time-of-flight method provides excellent energy resolution at energies up to several keV.  (See Firk \cite{Fi} for a review of time-of-flight spectrometers.)

\end{itemize}

The speed $v_n$ of a neutron can be determined by measuring the time $t_n$ that it takes to travel a measured distance $\ell$ in free space.  Using the standard result of special relativity, the kinetic energy of the neutron can be deduced using the equation \begin{eqnarray} E_n & \ = \ & E_0[(1 - v_n^2/c^2)^{-1/2} - 1] \nonumber\\ &=&  E_0[(1 - \ell^2/t_n^2c^2)^{-1/2} - 1], \end{eqnarray} where $E_0 \approx 939.553$ MeV is the rest energy of the neutron and $c \approx 2.997925 \cdot 10^8$ m/s is the speed of light.

If the units of energy are MeV, and those of length and time are meters and nanoseconds, then \begin{equation} 	E_n\  =\ 939.553[(1 - 11.126496 \ell^2/t_n^2)^{-1/2} - 1] {\rm \ MeV}. \end{equation} It is frequently useful to rearrange this equation to give the ratio $t_n/\ell$ for a given energy, $E_n$:  \begin{equation}\label{eq:firkthreepointthree}           t_n/\ell\  =\  3.3356404/ \sqrt{1 - (939.553/(E_n + 939.553))^2}. \end{equation}    Typical values for this ratio are 72.355 ns/m for $E_n = 1$ MeV and 23.044 ns/m for $E_n =  10$ MeV.

At energies below 1 MeV, the non-relativistic approximation to \eqref{eq:firkthreepointthree} is  adequate:   \begin{equation}       (t_n/\ell)_{\rm NR}\ =\ \sqrt{E_0/2 E_nc^2}\ =\ 72.298/\sqrt{E_n}\   \mu{\rm s/m}.	\end{equation} In the eV-region, it is usual to use units of $\mu$s/m: a $1$ eV neutron travels $1$ meter in $72.3$ microseconds.  At non-relativistic energies, the energy resolution $\Delta E$ at an energy $E$ is simply: \begin{equation} \Delta E\ \approx\ 2E \Delta t/t_E, \end{equation} where $\Delta t$ is the \emph{total} timing uncertainty, and $t_E$ is the flight time for a neutron of energy $E$.

In 1958, the two highest-resolution neutron spectrometers in the world had total timing uncertainties  $\Delta t \approx 200$ nanoseconds.  For a flight-path length of 50 meters the resolution was $\Delta E \approx 3$ eV at 1 keV.

In ${\ }^{238}{\rm U} + {\rm n}$, the excitation energy is about 5 MeV; the effective resolution for a 1 keV-neutron was therefore \begin{equation} \Delta E/E_{{\rm effective}} \ \approx \  6 \cdot 10^{-7} \end{equation} (at 1 eV, the effective resolution was about $10^{-11}$).

Two basic broadening effects limit the sensitivity of the method.

\begin{enumerate}

\item Doppler broadening of the resonance profile due to the thermal motion of the target nuclei; it is characterized by the quantity $\delta \approx 0.3\sqrt{E/A}$  (eV), where $A$ is the mass number of the target.  If $E = 1$ keV and $A = 200$, $\delta \approx 0.7$ eV, a value that may be ten times greater than the natural width of the resonance.

\item  Resolution broadening of the observed profile due to the finite resolving power of the spectrometer.  For a review of the experimental methods used to measure neutron total cross sections see Firk and Melkonian \cite{FMe}.  Lynn \cite{Ly} has given a detailed account of the theory of neutron resonance reactions.
\end{enumerate}

In the early 1950s, the field of low-energy neutron resonance spectroscopy was dominated by research groups working at nuclear reactors.  They were located at National Laboratories in the United States, the United Kingdom, Canada, and the former USSR.  The energy spectrum of fission neutrons produced in a reactor is moderated in a hydrogenous material to generate an enhanced flux of low-energy neutrons.  To carry out neutron time-of-flight spectroscopy, the continuous flux from the reactor is ``chopped'' using a massive steel rotor with fine slits through it.  At the maximum attainable speed of rotation (about $20,000$ rpm), and with slits a few thousandths-of-an-inch in width, it is possible to produce pulses each with a duration approximately 1 $\mu$sec.  The chopped beams have rather low fluxes, and therefore the flight paths are limited in length to less than 50 meters.  The resolution at $1$ keV is then $\Delta E \approx 20$ eV, clearly not adequate for the study of resonance spacings about 10 eV.

In 1952, there were only four accelerator-based, low-energy neutron spectrometers operating in the world.  They were at Columbia University in New York City, Brookhaven National Laboratory, the Atomic Energy Research Establishment, Harwell, England, and at Yale University. The performances of these early accelerator-based spectrometers were comparable with those achieved at the reactor-based facilities.  It was clear that the basic limitations of the neutron-chopper spectrometers had been reached, and therefore future developments in the field would require improvements in accelerator-based systems.

In 1956, a new high-powered injector for the electron gun of the Harwell electron linear accelerator was installed to provide electron pulses with very short durations (typically less than $200$ nanoseconds) \cite{FRG}.  The pulsed neutron flux (generated by the ($\gamma$, n) reaction) was sufficient to permit the use of a $56$ meter flight path; an energy resolution of $3$ eV at $1$ keV was achieved.

At the same time, Professors Havens and Rainwater (pioneers in the field of neutron time-of-flight spectroscopy) and their colleagues at Columbia University were building a new $385$ MeV proton synchrocyclotron a few miles north of the campus (at the Nevis Laboratory). The accelerator was designed to carry out experiments in meson physics and low-energy neutron physics (neutrons generated by the (p, n) reaction).  By 1958, they had produced a pulsed proton beam with duration of $25$ nanoseconds, and had built a $37$ meter flight path \cite{RDRH, DRRH}.  The hydrogenous neutron moderator generated an effective pulse width of about $200$ nanoseconds for $1$ keV-neutrons.  In 1960, the length of the flight path was increased to $200$ meters, thereby setting a new standard in neutron time-of-flight spectroscopy \cite{GRPH}.

%%%%%%%%%%%%%%%%%%%%%%%%%%%%%%%%%%%%%%%%%%%%%%%%%%%%%%%%%
%%%%%%%%%%%%%%%%%%%%%%%%%%%%%%%%%%%%%%%%%%%%%%%%%%%%%%%%%
%%%%%%%%%%%%%%%%%%%%%%%%%%%%%%%%%%%%%%%%%%%%%%%%%%%%%%%%%
\subsection{The Wigner Surmise}
At a conference on Neutron Physics by Time-of-Flight, held in Gatlinburg, Tennessee on November 1st and 2nd, 1956, Professor Eugene Wigner (Nobel Laureate in Physics, 1963) presented his surmise regarding the theoretical form of the spacing distribution of adjacent neutron resonances (of the same spin and parity) in heavy nuclei.  At the time, the prevailing wisdom was that the spacing distribution had a Poisson form (see, however, \cite{GP}).  The limited experimental data then available was not sufficiently precise to fix the form of the distribution (see \cite{Hu}).  The following quotation, taken from Wigner's presentation at the conference, introduces the concept of random matrices in Physics, for the first time: \begin{quote} Perhaps I am now too courageous when I try to guess the distribution of the distances between successive levels.  I should re-emphasize that levels that have different $J$-values (total spin) are not connected with each other.  They are entirely independent.  So far, experimental data are available only on even-even elements.  Theoretically, the situation is quite simple if one attacks the problem in a simple-minded fashion.  The question is simply `what are the distances of the characteristic values of a symmetric matrix with random coefficients?'

We know that the chance that two such energy levels coincide is infinitely unlikely. We consider a two-dimensional matrix,
$\left(\begin{array}{cc}
a_{11} & a_{12}  \\
a_{21} &  a_{22}
\end{array}\right)$, in which case the distance between two levels is $\sqrt{(a_{11} - a_{22})^2 + 4a_{12}^2}$.  This distance can be zero only if $a_{11} = a_{22}$ and $a_{12} = 0$. The difference between the two energy levels is the distance of a point from the origin, the two coordinates of which are $(a_{11} - a_{22})$ and $a_{12}$.  The probability that this distance is $S$ is, for small values of $S$, always proportional to $S$ itself because the volume element of the plane in polar coordinates contains the radius as a factor....

The probability of finding the next level at a distance $S$ now becomes proportional to $SdS$. Hence the simplest assumption will give the probability \begin{equation} \frac{\pi}{2} \rho^2 \exp\left( -\frac{\pi}{4}\rho^2S^2\right) SdS \end{equation}
for a spacing between $S$ and $S + dS$.

If we put $x = \rho S = S/\langle S \rangle$, where $\langle S \rangle$ is the mean spacing, then the probability distribution takes the standard form \begin{equation} p(x)dx\  = \  \frac{\pi}{2}\ x\ \exp\left(-\pi x^2/4\right)dx, \end{equation} where the coefficients are obtained by normalizing both the area and the mean to unity. \end{quote}

The form of the Wigner surmise had been previously discussed by Wigner \cite{Wig1}, and by Landau and Smorodinsky \cite{LS}, but not in the spirit of random matrix theory.
	
The Wigner form, in which the probability of zero spacing is zero, is strikingly different from the Poisson form \begin{equation} p(x)dx\ =\ \exp(-x)dx \end{equation} in which the probability is a maximum for zero spacing. The form of the Wigner surmise had been previously discussed by Wigner himself  \cite{Wig1}, and by Landau and Smorodinsky \cite{LS}, but not in the spirit of random matrix theory.

It is interesting to note that the Wigner distribution is a special case of a general statistical distribution, named after Professor E. H. Waloddi Weibull (1887-1979), a Swedish engineer and statistician \cite{Wei}.  For many years, the distribution has been in widespread use in statistical analyses in industries such as aerospace, automotive, electric power, nuclear power, communications, and life insurance.\footnote{In fact, one of the authors has used Weibull distributions to model run production in major league baseball, giving a theoretical justification for Bill James' Pythagorean Won-Loss formula \cite{Mil3}.}  The distribution gives the lifetimes of objects and is therefore invaluable in studies of the failure rates of objects under stress (including people!).  The Weibull probability density function is
\begin{equation} {\rm Wei}(x; k, \lambda)\ =\ \frac{k}{\lambda} \left(\frac{x}{\lambda}\right)^{k-1} \exp\left(-(x/\lambda)^k\right) \end{equation}
where $x \ge 0$, $k > 0$ is the \emph{shape} parameter, and $\lambda > 0$ is the \emph{scale} parameter.  We see that ${\rm Wei}(x; 2, 2/\sqrt{\pi}) = p(x)$, the Wigner distribution. Other important Weibull distributions are given in the following list.

\begin{itemize}
\item 	${\rm Wei} (x;1, 1) = \exp(-x)$ the Poisson distribution;

\item  ${\rm Wei} (x; 2, \lambda) = {\rm Ray}(\lambda)$, the Rayleigh distribution;

 \item ${\rm Wei} (x; 3, \lambda)$ is approximately a normal distribution.\footnote{Obviously this Weibull cannot be a normal distribution, as they have very different decay rates for large $x$, and this Weibull is a one-sided distribution! What we mean is that for $0 \le x \le 2$ this Weibull is well approximated by a normal distribution which shares its mean and variance, which are (respectively) $\Gamma(4/3) \approx .893$ and $\Gamma(5/3)-\Gamma(4/3)^2 \approx .105$.}

\end{itemize}

For ${\rm Wei}(x; k, \lambda)$, the mean is $\lambda \Gamma\left(1 + (1/k)\right)$, the median is $\lambda \log(2)^{1/k}$, and the mode is $\lambda(k - 1)^{1/k}/k^{1/k}$, if $k>1$. As $k \to \infty$, the Weibull distribution has a sharp peak at $\lambda$. Historically, Frechet introduced this distribution in 1927, and Nuclear Physicists often refer to the Weibull distribution as the Brody distribution \cite{BFFMPW}.

At the time of the Gatlinburg conference, no more than $20$ s-wave neutron resonances had been clearly resolved in a single compound nucleus and therefore it was not possible to make a definitive test of the Wigner surmise. Immediately following the conference, J. A. Harvey and D. J. Hughes \cite{HH}, and their collaborators, working at the fast-neutron-chopper-groups at the high flux reactor at the Brookhaven National Laboratory, and at the Oak Ridge National laboratory, gathered their own limited data, and all the data from neutron spectroscopy groups around the world, to obtain the first \emph{global spacing distribution} of s-wave neutron resonances.  Their combined results, published in 1958, showed a distinct lack of very closely spaced resonances, in agreement with the Wigner surmise.

By late 1959, the experimental situation had improved, greatly.  At Columbia University, two students of Professors Havens and Rainwater completed their Ph.D. theses; one, Joel Rosen \cite{RDRH}, studied the first $55$ resonances in ${\ }^{238}{\rm U} + {\rm n}$ up to $1$ keV, and the other, J Scott Desjardins \cite{DRRH}, studied resonances in two silver isotopes (of different spin) in the same energy region.  These were the first results from the new high-resolution neutron facility at the Nevis cyclotron.

At Harwell, Firk, Lynn, and Moxon \cite{FLM} completed their study of the first $100$ resonances in ${\ }^{238}{\rm U} + {\rm n}$ at energies up to $1.8$ keV; their measurement of the total neutron cross section for the interaction ${\ }^{238}{\rm U} + {\rm n}$  in the energy range $400$--$1800$ eV is shown in Figure \ref{fig:highresonancestudies}.

\begin{figure}[h]
\begin{center}
\scalebox{.7}{\includegraphics{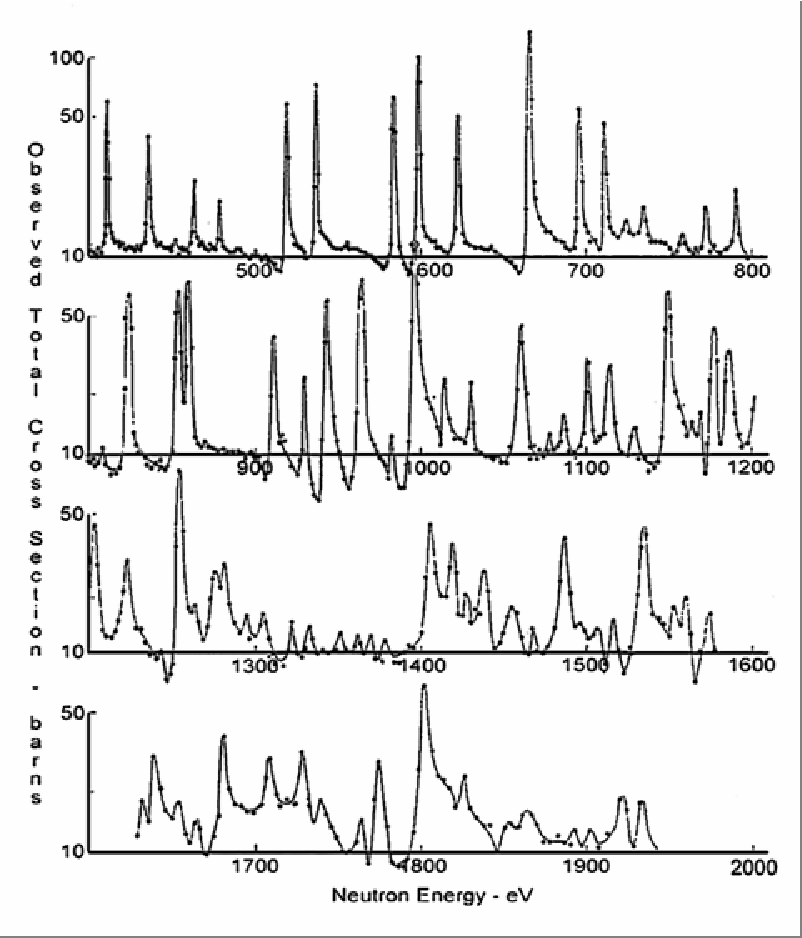}}
\caption{\label{fig:highresonancestudies} High resolution studies of the total neutron cross section of ${\ }^{238}{\rm U}$, in the energy range $400$ eV -- $1800$ eV. The vertical scale (in units of "barns") is a measure of the effective area of the target nucleus.}
\end{center}
\end{figure}

When this experiment began in 1956, no resonances had been resolved at energies above $500$ eV.  The distribution of adjacent spacings of the first $100$ resonances in the single compound nucleus, ${\ }^{238}{\rm U} + {\rm n}$, ruled out an exponential distribution and provided the best evidence (then available) in support of Wigner's proposed distribution.	

Over the last half-century, numerous studies have not changed the basic findings.  At the present time, almost $1000$ s-wave neutron resonances in the compound nucleus ${\ }^{239}{\rm U}$ have been observed in the energy range up to $20$ keV.  The latest results, with their greatly improved statistics, are shown in Figure \ref{fig:firkfig3} \cite{DLL}.	

\begin{figure}[h]
\begin{center}
\scalebox{.5}{\includegraphics{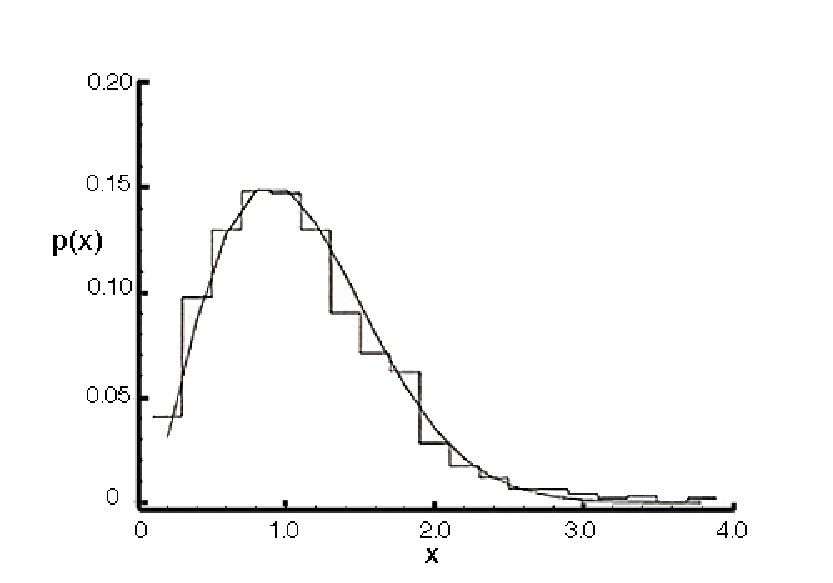}}
\caption{\label{fig:firkfig3} A Wigner distribution fitted to the spacing distribution of 932 s-wave resonances in the interaction ${\ }^{238}{\rm U} + {\rm n}$ at energies up to 20 keV.}
\end{center}
\end{figure}

%%%%%%%%%%%%%%%%%%%%%%%%%%%%%%%%%%%%%%%%%%%%%%%%%%%%%%%%%
%%%%%%%%%%%%%%%%%%%%%%%%%%%%%%%%%%%%%%%%%%%%%%%%%%%%%%%%%
%%%%%%%%%%%%%%%%%%%%%%%%%%%%%%%%%%%%%%%%%%%%%%%%%%%%%%%%%
\subsection{Some Nuclear Models}

It is interesting to note that, during the 1950s and 1960s, the study of the spacing distribution of neutron-induced resonances was far from the main stream of research in nuclear physics; almost all research was concerned with fundamental questions associated with nuclear structure and not with quantum statistical mechanics.  The newly-discovered Shell Model \cite{Ma, LE} of nuclei, and developments such as the Collective Model \cite{Ra, BM} were popular, and quite rightly so, when the successes of these models in accounting for the observed energies, spins and parities, and magnetic moments of nuclear states, particularly in light nuclei (mass numbers $< 20$, say) were considered.
	
These models were not able to account for the spacing distributions in heavy nuclei (mass numbers $A > 150$); the complex nature of so many strongly interacting nucleons prevented any detailed analysis.  However, the treatment of such complex problems had been considered in the mid-1930s, before the advent of the Shell-Model.  The Fermi Gas Model and other approaches based upon quantum versions of classical statistical mechanics and thermodynamics, were introduced, particularly by Bethe \cite{BE}.  The Fermi Gas Model treats the nucleons as non-interacting spin-$\frac12$ particles in a confined volume of nuclear size.  This, of course, seems at variance with the known strong interaction between pairs of nucleons.  However, the argument is made that the nuclear gas is completely degenerate and therefore, because of the Pauli exclusion principle, the nucleons can be considered free!  The model was the first to predict the energy-dependence of the density of states in the nuclear system.
	
The number of states that are available to a freely moving particle in a volume $V$ (the nuclear volume) that has a linear momentum in the range
$p$ to $p + dp$ is \begin{equation} dn\ =\ (4\pi V/h^3)p^2dp. \end{equation} This leads to \begin{equation}  n\ =\ (V/3\pi^2h^3) p_{\max}^3, \end{equation} where the result has been doubled because of the twofold spin degeneracy of the nucleons.  The ``Fermi energy'' $E_F$ corresponds to the maximum momentum: \begin{equation} E_F\ =\ p_{\max}^2/2m_{\rm Nucleon}. \end{equation} The level density $\rho(E^\ast)$ at an excitation energy $E^\ast$  predicted by the model is \begin{equation} \rho(E^\ast) \ =\  \rho(0) \exp\left(2\sqrt{aE^\ast}\right), \end{equation} where $a$ is given by the equation \begin{equation}
	E^\ast\  =\ a(kT)^2 \end{equation} in which $k$ is Boltzmann's constant and $T$ is the absolute temperature. The above expression for the level density is for states of all spins and parities.

In practical cases, $E^\ast$ is about 6 MeV for low- energy neutron interactions; this value leads to the following ratio for the mean level spacing at $E^\ast = 6$ MeV and at $E^\ast = 0$ (the ground state): \begin{equation} \langle D(6\ {\rm MeV})\rangle / \langle D(0)\rangle \ \approx \  4 \cdot 10^{-8}. \end{equation} For $\langle D(0) \rangle = 100$ keV (a practical value), the mean level spacing at  $E^\ast = 6$ MeV is $\approx 4 \cdot 10^{ -3}$ eV, which is more than three orders-of-magnitude smaller than typical values observed in heavy nuclei.
	
Many refinements of the model were introduced over the years; the models take into account spin, parity, and nucleon pairing effects.  A frequently used refined form is \begin{equation}\label{eq:firkeqthreepointonefour} \rho(E^\ast, J)\ =\ \rho(E^\ast, 0)(2J + 1)\exp\left(-(J(J + 1))/2\sigma^2\right), \end{equation} where $\sigma$ is called the ``spin-cut-off parameter''; the value of $\sigma^2$  is typically about $10$.  The predicted spacing distributions for two values of $\sigma$, and their comparison with a Wigner and an exponential distribution is shown in Figure \ref{fig:firkfig5}.

\begin{figure}
\begin{center}
\scalebox{.7}{\includegraphics{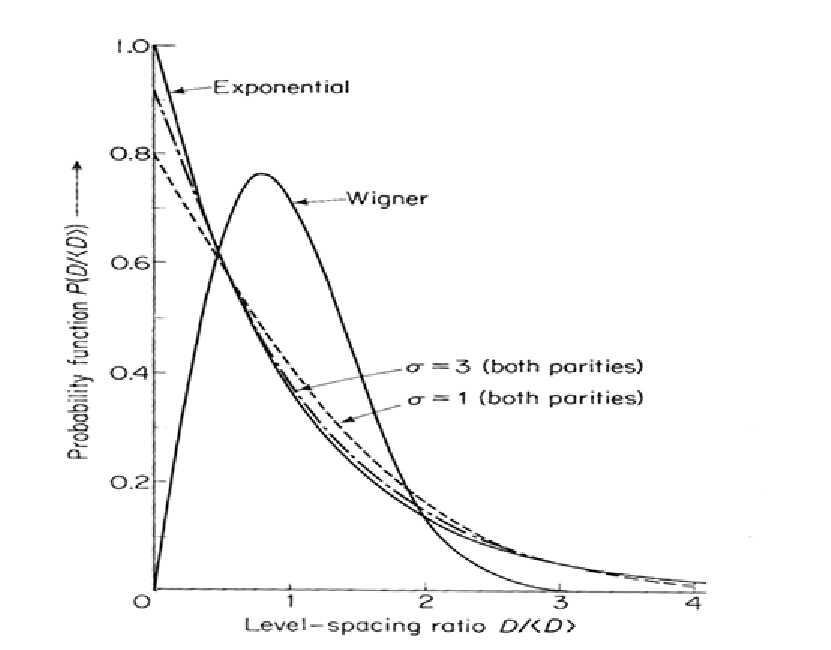}}
\caption{\label{fig:firkfig5} The spacing distribution of adjacent levels of the same spin and parity follows a Wigner distribution.  For a completely random distribution of levels (in both spin and parity) the distribution function is exponential. The distributions for random superpositions of several sequencies (each of which is of a Wigner form with a characteristic spin and parity) are, for level densities given by \eqref{eq:firkeqthreepointonefour} and $\sigma  = 1$ and $3$, found to approach the exponential distribution.}
\end{center}
\end{figure}

%%%%%%%%%%%%%%%%%%%%%%%%%%%%%%%%%%%%%%%%%%%%%%%%%%%%%%%%%
%%%%%%%%%%%%%%%%%%%%%%%%%%%%%%%%%%%%%%%%%%%%%%%%%%%%%%%%%
%%%%%%%%%%%%%%%%%%%%%%%%%%%%%%%%%%%%%%%%%%%%%%%%%%%%%%%%%
\subsection{The Optical Model}
	
In 1936, Ostrofsky et. al. \cite{OBJ} introduced a model of nuclear reactions that employed a complex nuclear potential to account for absorption of the incoming nucleon.  Later, Feshbach, Porter and Weisskopf \cite{FPW} introduced an important development of the model that helped further our understanding of the average properties of parameters used to describe nuclear reactions at low energies.

The following discussion provides insight into the physical content of their model.  Consider the plane-wave solutions of the Schr\"odinger equation: \begin{eqnarray}	 \frac{d^2\phi}{dx^2} + (2m/h^2)[E + V_0 + iW]\phi  \  =\  0, \ \ \ \ \ \ \phi \ = \ \exp(\pm ikx), \end{eqnarray} where the $+$ sign indicates outgoing waves and the $-$ sign indicates incoming waves. The wave number, $k$ is complex:  \begin{equation} k \ =\  \sqrt{(2m/h^2)[(E+V0) + iW]},  \end{equation} which can be written \begin{equation} 	k\ =\ k_{\rm R} + K_{\rm IM}. \end{equation} For $W < (E + V_0)$ (a reasonable assumption) we have  \begin{eqnarray} k_{\rm R} & \ =\ & 1/\lambda\  \approx \  \sqrt{(2m/h2)(E + V_0)} \nonumber\\  K_{\rm IM} &  = &  [W/(E + V_0)](k/2).\end{eqnarray}  Taking typical practical values $E = 10$ MeV, $V_0 = 40$ MeV and $W = 10$ MeV, the wave numbers are $k_{\rm R} \approx 1.5 {\rm fm}-1$ and $K_{\rm IM} \approx k_{\rm R}/10 \approx 0.15 {\rm fm}^{-1}$.

We see that the outgoing solution of the wave equation is \begin{equation} \phi \ =\ \exp\left(ik_{\rm R}x\right) \exp\left(-K_{\rm IM}x\right), \end{equation}
which represents an exponentially attenuated wave.  The wave number $K_{\rm IM}$ is effectively an attenuation coefficient.  The ``decay length'' associated with the probability function $|\phi|^2$ is the ``mean free path'':
\begin{equation}	\Lambda\ =\ 1/2K_{\rm IM}\ =\ (E + V_0)/Wk_{\rm R}. \end{equation} Using the above values for the energies, we obtain $\Lambda \approx 3.2{\rm fm}$. This value is of nuclear dimension, and supports the underlying hypothesis of the Compound Nucleus Model.
	
If the mean spacing of energy levels of a particle of mass $m$ inside the compound nucleus is $\langle D\rangle$, and its wave number is $K$, then the particle covers a distance \begin{equation} d \ \approx \ (h/\langle D\rangle)((hK/2\pi m)\ =\ (h^2K)/(2\pi m\langle D\rangle) \end{equation} inside the nucleus at an average speed  $\langle v\rangle \approx hK/2\pi m$ before it is emitted (or before another indistinguishable particle is emitted).  At an excitation energy of 10 MeV, a mean level spacing  $\langle D\rangle \approx 40$ eV, and a mean lifetime  $h/\langle D\rangle \approx 10^{- 16}$ sec are predicted.  These are reasonable values, considering the crudeness of the model.
	
The level density and level widths increase as the neutron bombarding energy increases; an energy region is therefore reached in which the levels completely overlap.  Cross section measurements then provide information on the average properties of the levels and, in particular, on the \emph{neutron strength function} \cite{LTW} defined as \begin{equation}  S\  =\  \langle \gamma_{\lambda n} \rangle^2 / \langle D \rangle \end{equation}
in which $=\  \langle \gamma_{\lambda n} \rangle^2$ is the average reduced neutron width and $\langle D \rangle$ is the average spacing.  For $s$-wave neutrons, $\gamma_{\lambda n}^2 = 2ka\Gamma_{\lambda n}$, where $k$ is the neutron wave number, $a$ is the nuclear radius, and $\Gamma_{\lambda n}$ is the neutron width of the level $\lambda$.
    	
The average absorption cross section $\langle \sigma_{\rm abs}\rangle$ may be obtained by averaging over the collision function $U$ \cite{LTW}.  The following expressions are then obtained: \begin{eqnarray}  1 - \left|\langle U \rangle \right|^2 & \ = \ & 2\pi \left(\langle \Gamma_{\lambda n} \rangle / \langle D\rangle\right) \nonumber\\ \langle \sigma_{\rm abs}\rangle  &=&  (\pi/k^2)g\left(1 - \left|\langle U \rangle \right|^2\right), \end{eqnarray} where $g$ is  a statistical ``spin weighting factor''.
    	
The term $1 - \left|\langle U \rangle \right|^2$ is directly related to the cross section for the formation of a compound nucleus \cite{FPW} which is, in turn, proportional to the strength function.  \emph{The importance of studying the spacing distribution of resonances, of a given spin and parity, originated in recognizing that the value of $\langle D\rangle$, the average spacing, appears as the denominator in the fundamental strength function. }

%%%%%%%%%%%%%%%%%%%%%%%%%%%%%%%%%%%%%%%%%%%%%%%%%%%%%%%%%
%%%%%%%%%%%%%%%%%%%%%%%%%%%%%%%%%%%%%%%%%%%%%%%%%%%%%%%%%
%%%%%%%%%%%%%%%%%%%%%%%%%%%%%%%%%%%%%%%%%%%%%%%%%%%%%%%%%
\subsection{Further Developments}

The first numerical investigation of the distribution of successive eigenvalues associated with random matrices was carried out by Porter and Rozenzweig in the late 1950s \cite{PR}. They diagonalized a large number of matrices where the elements are generated randomly but constrained by a probability distribution.  The analytical theory developed in parallel with their work: Mehta \cite{Meh1}, Mehta and Gaudin \cite{MG}, and Gaudin \cite{Gau}. At the time it was clear that the spacing distribution was not influenced significantly by the chosen form of the probability distribution. Remarkably, the $n \times n$ distributions had forms given \emph{almost exactly} by the original Wigner $2\times 2$ distribution.

The linear dependence of $p(x)$ on the normalized spacing $x$ (for small $x$) is a direct consequence of the \emph{symmetries} imposed on the Hamiltonian matrix, $H(h_{ij})$.  Dyson \cite{Dy1} discussed the general mathematical properties associated with random matrices and made fundamental contributions to the theory by showing that different results are obtained when different \emph{symmetries} are assumed for $H$.  He introduced three basic distributions; in Physics, only two are important, they are:

\begin{itemize}
\item	the Gaussian Othogonal Ensemble (GOE) for systems in which rotational symmetry and time-reversal invariance holds (the Wigner distribution): $p(x)$ $=$ $(\pi/2)$ $x$ $\exp\left(-(\pi/4)x^2\right)$;

\item the Gaussian Unitary Ensemble (GUE) for systems in which time-reversal invariance does not hold (French et. al. \cite{FKPT}):	$p(x) = (32/\pi^2) x^2\exp(-(\pi/4)x^2)$.

\end{itemize}
The mathematical details associated with these distributions are given in \cite{Meh1}.

The impact of these developments was not immediate in nuclear physics.  At the time, the main research endeavors were concerned with the structure of nuclei--experiments and theories connected with Shell-, Collective-, and Unified models, and with the nucleon-nucleon interaction.  The study of quantum statistical mechanics was far removed from the mainstream.   Almost two decades went by before random matrix theory was introduced in other fields of physics (see, for example, Bohigas, Giannoni and Schmit \cite{BGS} and Alhassid \cite{Al}).

%%%%%%%%%%%%%%%%%%%%%%%%%%%%%%%%%%%%%%%%%%%%%%%%%%%%%%%%%%%%%%%%%%%%%%%%%%%%%%%%%%%%%%%%%%%%%%%%%%%%%%%%%%%%%%%%%%%%%%%%%%%%%%%%%%%%
%%%%%%%%%%%%%%%%%%%%%%%%%%%%%%%%%%%%%%%%%%%%%%%%%%%%%%%%%%%%%%%%%%%%%%%%%%%%%%%%%%%%%%%%%%%%%%%%%%%%%%%%%%%%%%%%%%%%%%%%%%%%%%%%%%%%
%%%%%%%%%%%%%%%%%%%%%%%%%%%%%%%%%%%%%%%%%%%%%%%%%%%%%%%%%%%%%%%%%%%%%%%%%%%%%%%%%%%%%%%%%%%%%%%%%%%%%%%%%%%%%%%%%%%%%%%%%%%%%%%%%%%%
%%%%%%%%%%%%%%%%%%%%%%%%%%%%%%%%%%%%%%%%%%%%%%%%%%%%%%%%%%%%%%%%%%%%%%%%%%%%%%%%%%%%%%%%%%%%%%%%%%%%%%%%%%%%%%%%%%%%%%%%%%%%%%%%%%%%

\subsection{Lessons from Nuclear Physics}\label{sec:lessonsnuclearphys}

We have discussed at great length the connections between nuclear physics and number theory, with random matrix theory describing the behavior in these two very different fields. Before we analyze in great detail the success it has had in modeling the zeros of $L$-functions, it's worth taking a few moments to create a dictionary comparing these two subjects.

In nuclear physics the main object of interest is the nucleus. It is a many-bodied system governed by complicated forces. We are interested in studying the internal energy levels. To do so, we shoot neutrons (which have no net charge) at the nucleus, and observe what happens. Ideally we would be able to send neutrons of any energy level; unfortunately in practice we can only handle neutrons whose energies are in a certain band. The more energies at our disposal, the more refined an analysis is possible. Finally, there is a remarkable universality from heavy nucleus to heavy nucleus, where the distribution of spacings between adjacent energy levels depends weakly on the quantum numbers.

Interestingly, there are analogues of all these quantities on the number theory side. The nucleus is replaced by an $L$-function, which is built up as an Euler product of many factors of arithmetic interest. We are interested in the zeros of this function. We can glean information about them by using the explicit formula, \eqref{eq:genexpformula}. We first choose an even Schwartz test function $\phi$ whose Fourier transform $\widehat{\phi}$ has compact support. The explicit formula relates sums of $\phi$ at the zeros of the $L$-function to weighted sums of $\widehat{\phi}$ at the primes. Thus the more functions $\widehat{\phi}$ where we can successfully execute the sums over the primes, the more information we can deduce about the zeros. Unfortunately, in practice we can only evaluate the prime sums for $\widehat{\phi}$ with small support (if we could do arbitrary $\widehat{\phi}$, we could take a sequence converging to the constant function 1, whose inverse Fourier transform would be a delta spike at the origin and thus tell us what is happening there). Similar to the weak dependence on the quantum numbers, the answers for many number theory statistics depend weakly on the Satake parameters (whose moments are the Fourier coefficients in the series expansion of the $L$-function). In particular, the spacing between adjacent zeros is independent of the distribution of these parameters, though other statistics (such as the distribution of the first zero or first few zeros above the central point) fall into several classes depending on their distribution.

We collect these correspondences in the table below. While the structures studied in the two fields are very different, we can unify the presentations. In both settings we study the spacings between objects. While there are exact rules that govern their behavior, these are complicated. We gain information through interactions of test objects with our system; as we can only analyze these interactions in certain windows, we gain only partial information on the items of interest.  \\ \

\begin{center}
\begin{tabular}{lcccc}
   \textbf{Item} & \ \ \ \ \ \ & \textbf{Nuclear Physics}  & \ \ \ \ \ \ & \textbf{Number Theory}  \\
   \hline
   Object & \ \ \ \ \ \ & nucleus & \ \ \ \ \ \ & $L$-function  \\
   Events & \ \ \ \ \ \ & energy levels & \ \ \ \ \ \ & zeros \\
   Probe & \ \ \ \ \ \ & neutron (no net charge) & \ \ \ \ \ \ & test function $\phi$ (Schwartz) \\
   Restriction & \ \ \ \ \ \ & neutron's energy & \ \ \ \ \ \ & ${\rm supp}(\widehat{\phi})$ \\
   Individuality & \ \ \ \ \ \ & quantum numbers & \ \ \ \ \ \ & Satake parameters \\
  \hline
\end{tabular}
\end{center}

\ \\

We end by extracting some lessons from nuclear physics for number theory. The first is the importance of using the proper test function, or related to that the proper statistic. In the gold-foil experiments (1908 to 1913) positively charged alpha particles, which are helium nuclei, were used. Because they have a net positive charge, they are repelled by the nucleus they are probing. With the discovery of the neutron in 1932, physicists had a significantly better tool for studying the nucleus. As the machinery improved, more and more neutron energy levels were available, which led to sharper resolutions of the internal structure. We see variants of these on the number theory side, from restrictions on the test function to the consequences of increasing support. For example, when Wigner made his bold conjectures the data was not sufficiently detailed to rule out Poissonian behavior; that was not done until later when better experiments were carried out. Similar situations arise in number theory, where some statistics are consistent with multiple models and only by increasing the support are we able to determine the true underlying behavior. Finally, while there is a remarkable universality in behavior of the zeros, as for statistics such as adjacent spacings or $n$-level correlations the exact form of the $L$-function coefficients do not matter, these distributions do affect the rate of convergence to the random matrix theory predictions, as well as govern other statistics.

%%%%%%%%%%%%%%%%%%%%%%%%%%%%%%%%%%%%%%%%%%%%%%%%%%%%%%%%%%%%%%%%%%%%%%%%%%%%%%%%%%%%%%%%%%%%%%%%%%%%%%%%%%%%%%%%%%%%%%%%%%%%%%%%%%%%
%%%%%%%%%%%%%%%%%%%%%%%%%%%%%%%%%%%%%%%%%%%%%%%%%%%%%%%%%%%%%%%%%%%%%%%%%%%%%%%%%%%%%%%%%%%%%%%%%%%%%%%%%%%%%%%%%%%%%%%%%%%%%%%%%%%%
%%%%%%%%%%%%%%%%%%%%%%%%%%%%%%%%%%%%%%%%%%%%%%%%%%%%%%%%%%%%%%%%%%%%%%%%%%%%%%%%%%%%%%%%%%%%%%%%%%%%%%%%%%%%%%%%%%%%%%%%%%%%%%%%%%%%
%%%%%%%%%%%%%%%%%%%%%%%%%%%%%%%%%%%%%%%%%%%%%%%%%%%%%%%%%%%%%%%%%%%%%%%%%%%%%%%%%%%%%%%%%%%%%%%%%%%%%%%%%%%%%%%%%%%%%%%%%%%%%%%%%%%%
%%%%%%%%%%%%%%%%%%%%%%%%%%%%%%%%%%%%%%%%%%%%%%%%%%%%%%%%%%%%%%%%%%%%%%%%%%%%%%%%%%%%%%%%%%%%%%%%%%%%%%%%%%%%%%%%%%%%%%%%%%%%%%%%%%%%%%%%
%%%%%%%%%%%%%%%%%%%%%%%%%%%%%%%%%%%%%%%%%%%%%%%%%%%%%%%%%%%%%%%%%%%%%%%%%%%%%%%%%%%%%%%%%%%%%%%%%%%%%%%%%%%%%%%%%%%%%%%%%%%%%%%%%%%%%%%%
%%%%%%%%%%%%%%%%%%%%%%%%%%%%%%%%%%%%%%%%%%%%%%%%%%%%%%%%%%%%%%%%%%%%%%%%%%%%%%%%%%%%%%%%%%%%%%%%%%%%%%%%%%%%%%%%%%%%%%%%%%%%%%%%%%%%%%%%

\section[From Class Numbers to Pair Correlation \& Random Matrix Theory]{From Class Numbers to Pair Correlation and Random Matrix Theory}\label{sec:classtopaircorr}

The discovery that the pair correlation of the zeros of the Riemann zeta function (and other statistics of its zeros, and the zeros of other $L$-functions) are related to eigenvalues of random matrix ensembles has its beginnings with one of the most challenging problems in analytic number theory: the class number problem. Hugh Montgomery's investigation into the vertical distribution of the nontrivial zeros of $\zeta(s)$ arose during his work with Weinberger \cite{MW} on the class number problem. We give a short introduction to this problem to motivate Montgomery's subsequent work on the differences between zeros of $\zeta(s)$. We assume the reader is familiar with the basics of algebraic number theory and $L$-functions; an excellent introduction is Davenport's classic \emph{Multiplicative Number Theory} \cite{Da}. For those wishing a more detailed and technical discussion of the class number problem and its history, see \cite{Gol3, Gol4}. We then continue with a discussion of Montgomery's work on pair correlation, followed by the work of Odlyzko and others on spacings between adjacent zeros. After introducing the number theory motivation and results, we reveal the connection to random matrix theory, and conclude with a discussion of the higher level correlations, other related statistics, and open problems.

As there are too many areas of current research to describe them all in detail in a short article, we have chosen to concentrate on two major areas: the main terms for the $n$-level correlations, and the lower order terms; thus we do not describe many other important areas of research, such as the determination of moments or value distribution. The main terms are believed to be described by random matrix theory; however, the lower order terms depend on subtle arithmetic of the $L$-functions, and there we can see different behavior. The situation is very similar to that of the Central Limit Theorem, and we will describe these connections and viewpoints in greater detail below.

%%%%%%%%%%%%%%%%%%%%%%%%%%%%%%%%%%%%%%%%%%%%%%%%%%%%%%%%%%%%%%%%
%%%%%%%%%%%%%%%%%%%%%%%%%%%%%%%%%%%%%%%%%%%%%%%%%%%%%%%%%%%%%%%%
%%%%%%%%%%%%%%%%%%%%%%%%%%%%%%%%%%%%%%%%%%%%%%%%%%%%%%%%%%%%%%%%
\subsection{The Class Number Problem}\label{sec:classnumberproblem}

Let $K=\mathbb{Q}(\sqrt{-q})$ be the imaginary quadratic field associated to the negative fundamental discriminant $-q$. Here we have that $-q$ is congruent to 1 $(\text{mod }4)$ and square-free or $-q=4m$, where $m$ is congruent to 2 or 3 $(\text{mod }4)$ and square-free. The class number of $K$, denoted $h(-q)$, is the size of the group of ideal classes of $K$. When $h(-q)=1$,  the ring of integers of $K$, denoted $\mathcal{O}_K$, has unique factorization. Such an occurrence (the class number one problem, discussed below) is rare, and the class number $h(-q)$ may be thought of as a measure on the failure of unique factorization in $\mathcal{O}_K$.

One of the most difficult problems in analytic number theory is to estimate the size of $h(-q)$ effectively. Gauss \cite{Ga} showed that $h(-q)$ is finite and further conjectured that $h$ tends to infinity as $-q$ runs over the negative fundamental discriminants. This conjecture was proved by Heilbronn \cite{He} in 1934. Thus, while it is settled that there are only finitely many imaginary quadratic fields with a given class number $h(-q)$, an obvious question remains: can we list all imaginary quadratic fields $K$ with a given class number $h(-q)$? This is the class number problem.

One may easily deduce an upper bound on $h(-d)$ via Dirichlet's class number formula. For $\Re(s)>1$, let $L(s,\chi_{-q})$ denote the Dirichlet $L$-function
\begin{equation}
L(s,\chi_{-q})\ := \ \sum_{n=1}^{\infty}\frac{\chi_{-q}(n)}{n^s},
\end{equation}
where $\chi_{-q}(n)$ is the Kronecker symbol associated to the fundamental discriminant $-q$. In order to prove the equidistribution of primes in arithmetic progression, Dirichlet derived the class number formula,
\begin{equation}\label{classnumberformula}
h(-q)\ = \ \frac{w\sqrt{q}}{2\pi}L(1,\chi_{-q}),
\end{equation}
where $w$ denotes the number of roots of unity of $K=\mathbb{Q}(\sqrt{-q})$:
\begin{equation}
w \ = \  \begin{cases}
2 & \text{if } q>4\\
4 & \text{if } q=4\\
6 & \text{if } q=3.\\
\end{cases}
\end{equation} Dirichlet needed to show $L(1,\chi_{-q}) \neq 0$, which is immediate from the class number formula as $h(-q) \ge 1$. This connection between class numbers and zeros of $L$-functions is almost 200 years old, and illustrates how knowledge of zeros of $L$-functions yields information on a variety of important problems.

Instead of using the class number formula to prove non-vanishing of $L$-functions, we can use results on the size of $L$-functions to obtain bounds on the class number. Combining \eqref{classnumberformula} with that fact that $L(1,\chi_{-q})\ll \log q$, it follows that $h(-q) \ll \sqrt{q}\log q$. On the other-hand, Siegel \cite{Sie2} proved that for every $\varepsilon>0$ we have $L(1,\chi_{-q}) > c(\varepsilon)q^{-\varepsilon}$, where $c(\varepsilon)$ is a constant depending on $\varepsilon$ that is not numerically computable for small $\varepsilon$. Upon inserting this lower bound in \eqref{classnumberformula}, it follows that $h(-q) \gg c(\varepsilon)q^{1/2-\varepsilon}$; however this does not help us solve the class number problem because the implied constant is ineffective.\footnote{In other words, while the above is enough to prove that the class number tends to infinity, we cannot use that argument to produce an explicit constant $Q_n$ for each $n$ so that we could assert that the class number is at least $n$ if $q \ge Q_n$. One of the best illustrations of the importance of \emph{effective} constants is the following joke: There is a constant $T_0$ such that if all the non-trivial zeros of $\zeta(s)$ in the critical strip up to height $T_0$ are on the critical line, then they all are and the Riemann Hypothesis is true; in other words, it suffices to check up to a finite height! To see this, if the Riemann Hypothesis is true we may take $T_0$ to be 0, while if it is false we take $T_0$ to be 1 more than the height of the first exemption. We have therefore shown a constant exists, but such information is completely useless!} Computing an effective lower bound on $h(-q)$ is very difficult task.

The class number one problem was eventually solved independently by Heegner \cite{Hee}, Stark \cite{St1} and Baker \cite{Ba1}. For $h(-d)=2$, the class number problem was solved independently by Stark \cite{St2}, Baker \cite{Ba2} and Montgomery and Weinberger \cite{MW}. In 1976, Goldfeld \cite{Gol1,Gol2} showed that if there exists an elliptic curve $E$ whose Hasse-Weil $L$-function has a zero at the central point $s=1$ of order at least three, then for any $\varepsilon>0$, we have $h(-q) > c_{\varepsilon,E}\log(|-q|)^{1-\varepsilon}$, where the constant $c_{\varepsilon,E}$ is effectively computable. In other words, Goldfeld proved that if there exists an elliptic curve whose Hasse-Weil $L$-function has a triple zero at $s=1$, then the class number problem is reduced to a finite amount of computations.  In 1983, Gross and Zagier \cite{GZ} showed the existence of such an elliptic curve. Combining this deep work of Gross-Zagier with a simplified version of Goldfield's argument to reduce the amount of necessary computations, Oesterl\'{e} \cite{Oe} produced a complete list of imaginary quadratic fields with $h(-q)=3.$ To date, the class number problem is resolved for all $1\le h(-q) \le 100$. (In addition to the previous references, see Arnon \cite{Ar}, Arnon, Robinson, and Wheeler \cite{ARW}, Wanger \cite{Wan} and Watkins \cite{Wa}.)

Combining their work with results of Stark \cite{St2} and Lehmer, Lehmer, and Shanks \cite{LLS}, Montgomery and Weinberger gave a complete proof for the class number two problem. Their proof is based on the curious Deuring-Heilbronn phenomenon, which implies that if $h(-d)<d^{1/4-\delta}$ then the low-lying nontrivial zeros of many quadratic Dirichlet $L$-functions are on the critical line, at least up to some height depending on $d$, $\delta$, and the $L$-functions. For an overview of the Deuring-Heilbronn phenomenon, see the survey article by Stopple \cite{Sto1}. Montgomery and Weinberger also establish that if the class number is a bit smaller, then one can show that these nontrivial zeros on the critical line are very evenly spaced. Moreover, more precise information about the vertical distribution of these zeros would imply an effective lower bound on $h(-d)$. Montgomery and Weinberger write: \begin{quote} Let $\rho=1/2+i\gamma$ and $\rho^{\prime}=1/2+i\gamma^{\prime}$ be consecutive zeros on the critical line of an $L$-function $L(s,\chi)$, where $\chi$ is a primitive character $(\text{mod }k)$. Put
\begin{equation}
\lambda(K)\ = \ \min\frac{1}{2\pi}|\gamma-\gamma^{\prime}|\log K,
\end{equation}
where the minimum is over all $k\le K$, all $\chi \,(\text{mod } k)$, and all $\rho=1/2+i\gamma$ of $L(s,\chi)$ with $|\gamma|\le 1$. In this range the average of $|\gamma-\gamma^{\prime}|$ is $2\pi/\log k$, so trivially $\limsup \lambda(K)\le 1$. Presumably $\lambda(K)$ tends to 0 as $K$ increases; if this could be shown effectively then the effective lower bound $h>d^{1/4-\varepsilon}$ would follow. In fact the weak inequality $\lambda(K)<1/4-\delta$ for $K>K_0$ implies that $h>d^{(1/2)\delta-\varepsilon}$ for $d>C(K_0,\varepsilon)$; the function $C(K_0,\varepsilon)$ can be made explicit. Even $\lambda(K)<\frac{1}{2}-\delta$ has striking consequences.\end{quote}

%%%%%%%%%%%%%%%%%%%%%%%%%%%%%%%%%%%%%%%%%%%%%%%%%%%%%%%%%%%%%%%%
%%%%%%%%%%%%%%%%%%%%%%%%%%%%%%%%%%%%%%%%%%%%%%%%%%%%%%%%%%%%%%%%
%%%%%%%%%%%%%%%%%%%%%%%%%%%%%%%%%%%%%%%%%%%%%%%%%%%%%%%%%%%%%%%%
\subsection{Montgomery's Pair Correlation of the Zeros of $\zeta(s)$}\label{sec:paircorrzeroszeta}

We have seen that the class number problem is related to another very difficult question in analytic number theory: \emph{What is the vertical distribution of the zeros of the Riemann zeta function (and general $L$-functions) on the critical line?}

Given an increasing sequence $\{\alpha_n\}_{n=1}^\infty$ and a box $B \subset \mathbb{R}^{n-1}$, the $n$-level correlation is defined by \begin{eqnarray}\label{eqnlevelcorr}
\lim_{N \rightarrow \infty} \frac{\#\left\{\left(\alpha_{j_1}-\alpha_{j_2}, \dots, \alpha_{j_{n-1}} - \alpha_{j_n}\right) \in B, j_i \neq j_k  \right\}}{N}.
\end{eqnarray}
The pair correlation is the case $n=2$, and through combinatorics knowing all the correlations yields the spacing between adjacent events (see for example \cite{Meh2}). In 1973, Montgomery \cite{Mon} was able to partially determine the behavior for the pair correlation of zeros of the Riemann zeta function, $\zeta(s)$, which led to new results on the number of simple zeros of $\zeta(s)$ and the existence of gaps between zeros of $\zeta(s)$ that are closer together than the average. One of the most striking contributions in Montgomery's paper, however, is his now famous pair correlation conjecture. We first state his conjecture and then discuss related work on spacings between adjacent zeros in the next subsection; after these have been described in detail we then revisit these problems and describe the connections with random matrix theory in \S\ref{sec:earlyyearsNTandRMT}. See \cite{CI} for more on connections between spacings of zeros of $\zeta(s)$ and the class number.\\ \

\begin{center}
\fbox{
\begin{minipage}{\textwidth}
\begin{conjecture}[Montgomery's pair correlation conjecture]\label{conj:pcc}
Assume the Riemann hypothesis, and let $\gamma,\gamma^{\prime}$ denote the imaginary parts of nontrivial zeros of $\zeta(s)$. For fixed $0<a<b<\infty$,
\begin{eqnarray}\label{paircorrelationconjecture}
& & \lim_{T\rightarrow \infty }\frac{\#\{\gamma,\gamma^{\prime}: 0 \le \gamma, \gamma^{\prime}\le T, 2\pi a(\log T)^{-1}\le \gamma-\gamma^{\prime}\le2\pi b(\log T)^{-1}\}}{\frac{T}{2\pi}\log T}\nonumber\\ & &  = \  \int_{a}^{b}1-\left(\frac{\sin \pi u}{\pi u}\right)^2\, du.
\end{eqnarray}
\end{conjecture}
Thus Montgomery's pair correlation conjecture is the statement that the pair correlation of the zeros of $\zeta(s)$ is
\begin{equation}
1-\left(\frac{\sin \pi u}{\pi u}\right)^2.
\end{equation}
\end{minipage}
}
\end{center}

\ \\

Notice that the factor $1-(\sin \pi u /  \pi u)^2$ suggests a `repulsion' between the zeros of $\zeta(s)$. The notion that the zeros cannot be too close to one another was also revealed in the aforementioned work of Montgomery and Weinberger as a consequence of the Deuring-Heilbronn phenomenon.

To arrive at his conjecture, Montgomery introduced the function
\begin{equation}\label{eq:MontFxT} F(x,T)\ = \ \sum_{0<\gamma,\gamma^\prime\leq T}x^{i(\gamma-\gamma^\prime)}w(\gamma-\gamma^\prime),\end{equation}
where $w(u)$ is a weight function given by $w(u) = 4/(4+u^2)$.
Let $F(\alpha)$ denote $F(x,T)$ with $x$ set as $x=T^\alpha$; then
\begin{equation}\label{MontF}
F(\alpha) \ = \  F(\alpha,T) \ = \  \left(\frac{T}{2\pi}\log T\right)^{-1}\sum_{0 \le \gamma,\gamma^{\prime}\le T}T^{i\alpha(\gamma - \gamma^{\prime})}w(\gamma-\gamma^\prime),
\end{equation}
where $\alpha$ and $T\ge 2$ are real. $F(\alpha)$ is a real, even function.
Let $r(u) \in L^1$, and define its Fourier transform by
\begin{equation}\label{ftrans}
\hat{r}(\alpha) \ = \  \int_{-\infty}^{\infty} r(u)e^{2\pi i \alpha u}du.
\end{equation}
The function $r$ is a test function that replaces the `box' in the statement of the pair correlation conjecture~\ref{conj:pcc}.
One notable item about Montgomery's pair correlation conjecture is that there is no
restriction on the length of the interval $[a,b]$; the difference $b-a$ is permitted to be
arbitrarily small. In the language of smooth test functions, this translates to
permitting arbitrarily large support on the Fourier transform $\hat r$.

If $\hat{r}(\alpha) \in L^1$, then upon multiplying \eqref{MontF} by $\hat{r}(\alpha)$ and integrating, we deduce
\begin{equation}\label{MontR}
\sum_{0<\gamma,\gamma^{\prime} \le T} r\left(\frac{(\gamma^{\prime}\!-\!\gamma)\log T}{2\pi}w(\gamma^{\prime}\!-\!\gamma) \right)\ \sim\ \left(\frac{T}{2\pi}\log T\right)\int_{-\infty}^{\infty}\hat{r}(\alpha)F(\alpha)d\alpha
\end{equation}
as $T$ tends to infinity. If the Riemann hypothesis is
true, the asymptotic~\eqref{MontR} connects the pair correlation of $\zeta(s)$ to the
function $F(\alpha)$ given in~\eqref{MontF}. Montgomery proceeded to prove an important
special case of Conjecture~\ref{conj:pcc} for a class of test functions with Fourier
transform supported in $(-1,1)$.

\begin{center}
\fbox{
\begin{minipage}{\textwidth}
\begin{theorem}[Montgomery's theorem]\label{thm:mont}
  Assume the Riemann hypothesis. For real $\alpha$, $T\geq2$, let $F(\alpha)$ be defined
  by~\eqref{MontF}. Then $F(\alpha)$ is real, and $F(\alpha)=F(-\alpha)$.
  If $T>T_0(\epsilon)$ then $F(\alpha)\geq-\epsilon$ for all $\alpha$. For fixed $\alpha$
  satisfying $0\leq\alpha<1$ we have
  \begin{equation}\label{MontgomeryF}
    F(\alpha) \ = \  \alpha + o(1) + T^{-2\alpha}\log T(1+o(1))
  \end{equation}
  uniformly for $0 \le \alpha < 1$ as $T$ tends to infinity.
\end{theorem}
\end{minipage}
}
\end{center}

Thus, for any function $r(u) \in L^{1}$ with Fourier transform $\hat{r}(\alpha)$ supported in $(-1,1)$, one can use \eqref{MontgomeryF} to evaluate the sums appearing in \eqref{MontR}.
For $\alpha\ge1$, Montgomery further conjectured, with heuristic arithmetic justification,
that
\begin{equation}\label{eq:spcc}
F(\alpha)\ =\ 1+o(1)\quad\text{uniformly in bounded intervals as }T\rightarrow\infty.\end{equation}
This conjecture, combined with \eqref{MontgomeryF} gives a complete picture of the function
$F(\alpha)$, which led Montgomery to make his pair correlation conjecture.

%%%%%%%%%%%%%%%%%%%%%%%%%%%%%%%%%%%%%%%%%%%%%%%%%%%%%%%%%%%%%%%%
%%%%%%%%%%%%%%%%%%%%%%%%%%%%%%%%%%%%%%%%%%%%%%%%%%%%%%%%%%%%%%%%
%%%%%%%%%%%%%%%%%%%%%%%%%%%%%%%%%%%%%%%%%%%%%%%%%%%%%%%%%%%%%%%%
\subsection[Proof of Pair Correlation Conjecture for Restricted $a, b$]{Proof of Montgomery's Pair Correlation Conjecture for Restricted $a, b$}

We now provide greater detail about Montgomery's original proof~\cite[\S3, pp. 187--191]{Mon}
of his theorem (Theorem~\ref{thm:mont}). The point of entry is an explicit formula due to
him.

The role of explicit formul\ae\ cannot be overstated when working with $\zeta(s)$ or
$L$-functions, as these formul\ae\ unlock the multiplicative structure implicit in the
Euler product, usually via the argument principal applied to the logarithmic derivative.
Assuming the Riemann hypothesis, and writing critical zeros of $\zeta(s)$ as
$1/2+i\gamma$ and $\gamma$ real, with $1<\sigma<2$ and $x\geq1$, Montgomery proved that
\begin{equation}\label{eq:explicit}\begin{aligned}[b]&(2\sigma-1)\sum_\gamma
  \frac{x^{i\gamma}}{\left(\sigma-\frac12\right)^2+\left(t-\gamma\right)^2} \\
  &\hspace{0.25in}\begin{aligned}[b]
    =\ &-x^{-1/2}\left(\sum_{n\leq x}\Lambda(n)\left(\frac xn\right)^{1-\sigma+it}
    +\sum_{n>x}\Lambda(n)\left(\frac xn\right)^{\sigma+it}\right) \\
  &+x^{1/2-\sigma+it}(\log\tau+O_\sigma(1))+O_\sigma(x^{1/2}\tau^{-1}),
\end{aligned}\end{aligned}\end{equation}
where $\tau=|t|+2$ and the implied constants depend only on $\sigma$.

\begin{proof}[Proof of Montgomery's theorem (Theorem~\ref{thm:mont}); {\cite[\S3, pp. 187--191]{Mon}}]
Placing $\sigma=3/2$ in~\eqref{eq:explicit}, and letting $L(x,t)$ and $R(x,t)$ denote
the left and right sides, respectively, we now wish to evaluate the second moments
of both sides; i.e.
$\int_0^T\left|L(x,t)\right|^2dt$, $\int_0^T\left|R(x,t)\right|^2dt$.
The reason to do this is that, as we will see, $F(\alpha)$ falls out of the second moment
of the left side, and we end up with something tractable for the second moment of the
right side. Thus the equation of the two moments gives us an identity for $F(\alpha)$.

By showing the contribution of those ordinates $\gamma$ above height $T$ is $O(\log^3 T)$,
Montgomery obtained
\begin{equation}\int_0^T\left|L(x,t)\right|^2dt\ = \ 4\sum_{\substack{0<\gamma\leq T \\
0<\gamma'\leq T}}x^{i(\gamma-\gamma')}\int_0^T\frac{dt}{(1+(t-\gamma)^2)(1+(t-\gamma')^2)}
+O(\log^3 T).\end{equation} Note that the range of integration may be extended to all of $\mathbb{R}$ at a
penalty no greater in magnitude than $O(\log^2 T)$; we then have
\begin{equation}\int_0^T\left|L(x,t)\right|^2dt\ = \  4\sum_{\substack{0<\gamma\leq T \\
    0<\gamma'\leq T}}x^{i(\gamma-\gamma')}
\int_{-\infty}^\infty\frac{dt}{(1+(t-\gamma)^2)(1+(t-\gamma')^2)}+O(\log^3 T);\end{equation}
it then follows from the residue calculus that the definite integral evaluates to
$w(\gamma-\gamma')\pi/2$ and
\begin{equation}\int_0^T\left|L(x,t)\right|^2dt\ = \ 2\pi\sum_{\substack{0<\gamma\leq T \\
    0<\gamma'\leq T}}x^{i(\gamma-\gamma')}w(\gamma-\gamma')+O(\log^3 T).\end{equation}
Putting $x=T^\alpha$ yields
\begin{equation}\label{montlhs}\int_0^T\left|L(x,t)\right|^2dt\ = \ F(\alpha)T\log T+O(\log^3 T).\end{equation}
The non-negativity of the left side of~\eqref{montlhs} gives the statement in Theorem~\ref{thm:mont}
of the positivity of $F(\alpha)$. (The evenness of $F(\alpha)$ follows from the fact that
$\gamma$ and $\gamma'$ may be interchanged in the definition~\eqref{MontF}.)
It then falls to evaluate $\int_0^T\left|R(x,t)\right|^2dt$. First,
\begin{equation}\int_0^T\left|x^{-1+it}\log\tau\right|^2dt\ = \ \frac T{x^2}(\log^2 T+O(\log T))\end{equation}
for all $x\geq1,T\geq2$. Montgomery then applied a quantitative version of Parseval's
identity for Dirichlet series to find \begin{equation}\label{eq:parseval}
\int_0^T\left|\sum_n a_nn^{-it}\right|^2dt\ = \ \sum_n\left|a_n\right|^2(T+O(n)).\end{equation}
Applying \eqref{eq:parseval}  to the explicit formula~\eqref{eq:explicit}, we find
\begin{equation}\label{eq:MontRHSPNT}\begin{aligned}[b]
  &\frac1x\int_0^T\left|\sum_{n\leq x}\Lambda(n)\left(\frac xn\right)^{-1/2+it}
    +\sum_{n>x}\Lambda(n)\left(\frac xn\right)^{3/2+it}\right|^2dt \\
  &\hspace{0.25in}=\ \frac1x\sum_{n\leq x}\Lambda(n)^2\left(\frac xn\right)^{-1}(T+O(n))
  +\frac1x\sum_{n>x}\Lambda(n)^2\left(\frac xn\right)^3(T+O(n)) \\
  &\hspace{0.25in}=\ T(\log x+O(1))+O(x\log x),
\end{aligned}\end{equation}
where the last line follows from the prime number theorem with error term.
It then follows from simple estimation of the error terms and a more delicate application
of Cauchy-Schwarz that
\begin{equation}\label{eq:montrhs}\int_0^T\left|R(T^\alpha,t)\right|^2dt
\   =\ ((1+o(1))T^{-2\alpha}\log T+\alpha+o(1))T\log T,\end{equation} uniformly for
$0\leq\alpha\leq1-\epsilon$. Combining~\eqref{montlhs} and~\eqref{eq:montrhs} yields
Montgomery's theorem.\qed
\end{proof}

We end this section by describing the heuristic evidence that led Montgomery to
conjecture~\eqref{eq:spcc} on the behavior of $F(\alpha)$ for $\alpha>1$.
The argument above for proving Montgomery's conjecture for $0\leq\alpha<1$ fails for
$\alpha>1$, since error terms such as in~\eqref{eq:MontRHSPNT} and those arising from
Cauchy-Schwarz and the last line of~\eqref{eq:explicit} are no longer dominated by the
main term.

Examining the sum over primes from the explicit formula~\eqref{eq:explicit} with
$\sigma=3/2$,
\begin{equation}\sum_{n\leq x}\Lambda(n)\left(\frac xn\right)^{-1/2+it}+\sum_{n>x}\Lambda(n)\left(\frac
xn\right)^{3/2+it},\end{equation} the expected value is seen by the prime number theorem to be
\begin{equation}\frac{2x^{1-it}}{\left(\frac12+it\right)\left(\frac32-it\right)}.\end{equation} From the
proof of Montgomery's theorem we have, with $F(x,T)$ as in~\eqref{eq:MontFxT}, that
\begin{equation}\begin{aligned}[b]
  F(x,T)\ = \ \frac1{2\pi x}\int_0^T\Bigg|&\sum_{n\leq x}\Lambda(n)
  \left(\frac xn\right)^{-1/2+it}+\sum_{n>x}\Lambda(n)\left(\frac xn\right)^{3/2+it} \\
  &\hspace{0.25in}-\frac{2x^{1-it}}{\left(\frac12+it\right)\left(\frac32-it\right)}\Bigg|^2dt+o(T\log T);\end{aligned}\end{equation}
it follows that we would like to know the size of
\begin{equation}\int_0^T\left|\frac1x\sum_{n\leq x}\Lambda(n)n^{1/2-it}+x\sum_{n>x}\Lambda(n)n^{-3/2-it}
-\frac{2x^{1/2-it}}{\left(\frac12+it\right)\left(\frac32-it\right)}\right|^2dt.\end{equation}
Montgomery proceeded to multiply out and integrate term-by-term, finding that the
non-diagonal is non-neglectable. He collected terms in the form of sums of the sort
\begin{equation}\label{eq:lamlamsum}\sum_{n\leq y}\Lambda(n)\Lambda(n+h);\end{equation}
invoking the Hardy-Littlewood $k$-tuple conjecture for 2-tuples with a strong error
term, \eqref{eq:lamlamsum} should be $\asymp y$. This would give
\begin{equation} F(x,T)\ \sim\ \frac T{2\pi}\log T\end{equation} in $x\leq T\leq x^{2-\epsilon}$, and there is little
reason to expect the behavior to change for bounded $\alpha\geq2$. On this basis,
Montgomery made his conjecture~\eqref{eq:spcc}.

%%%%%%%%%%%%%%%%%%%%%%%%%%%%%%%%%%%%%%%%%%%%%%%%%%%%%%%%%%%%%%%%
%%%%%%%%%%%%%%%%%%%%%%%%%%%%%%%%%%%%%%%%%%%%%%%%%%%%%%%%%%%%%%%%
%%%%%%%%%%%%%%%%%%%%%%%%%%%%%%%%%%%%%%%%%%%%%%%%%%%%%%%%%%%%%%%%
\subsection{Spacings Between Adjacent Zeros}\label{sec:spacingsbetweenzeros}

Motivated by Montgomery's pair correlation conjecture on the zeros of the Riemann zeta function, starting in the late 1970s Andrew Odlyzko began a large-scale computation of zeros of $\zeta(s)$ high in the critical strip. The average spacing between zeros of $\zeta(s)$ at height $T$ in the critical strip is on the order of $1/\log T$; thus as we go higher and higher we have more and more zeros in regions of fixed size, and there is every reason to hope that, after an appropriate normalization, a limiting behavior exists.

The story of computing zeta zeros goes back to Riemann himself. As mentioned in \S\ref{sec:Lfunc_and_zeros}, in his one paper on the zeta function \cite{Ri}, Riemann states the Riemann hypothesis (RH) in passing. He used a formula now known as the Riemann-Siegel formula to compute a few zeros of $\zeta(s)$ up to a height of probably no greater than 100 in the critical strip; though he did not mention these computations in the paper, the role of these computations was important in the development of mathematics and mirror the role played by the calculation of energy levels in nuclear physics in illuminating the internal structure of the nucleus. The formula was actually lost for almost 70 years, and did not enter the mathematics literature until Siegel was reading Riemann's works \cite{Sie1}. Siegel's role in understanding, collecting, and interpreting Riemann's notes should not be underestimated, since the expertise and insight needed to infer the ideas behind the notes was great.

% Pi(s) = Gamma(s+1)
% Gamma(s) = Pi(s-1)
The development of the Riemann-Siegel formula proceeds along the purely classical lines
of complex analysis. Riemann had a formula for $\zeta(s)$ valid for all $s\in\mathbb{C}$; namely,
\begin{equation}\zeta(s)
\ = \ \frac{\Gamma(1-s)}{2\pi i}\int_{\mathcal C}\frac{(-x)^s}{e^x-1}\cdot\frac{dx}x,\end{equation}
where $\mathcal C$ is the contour that starts at $+\infty$, traverses the real axis towards
the origin, circles the origin once with the positive orientation about $0$, and then
retraces its path along the real axis to $+\infty$.

By splitting off some finite sums from the contour integral above, Riemann arrived at
the formula
\begin{eqnarray}\zeta(s) & \ =\ &   \sum_{n=1}^N\frac1{n^s} \ + \ \pi^{1/2-s}
    \frac{\Gamma\left(\tfrac s2\right)}{\Gamma\left(\tfrac12(1-s)\right)}
    \sum_{n=1}^M\frac1{n^{1-s}}\nonumber\\
    & & \ \ \ - \ \frac{\Gamma(1-s)}{2\pi i}
    \int_{\mathcal C_M}\frac{(-x)^s e^{-Nx}}{e^x-1}\cdot\frac{dx}x,
\end{eqnarray} where here $s\in\mathbb{C}$, $N,M\in\mathbb{N}$ are arbitrary, and $\mathcal C_M$ is the contour that traces from $+\infty$ to $(2M+1)\pi$, circles along $|s|=(2M+1)\pi$ once with positive orientation, and then returns to $+\infty$, thereby enclosing the poles $\pm2\pi iM,\pm2\pi i(M-1),\ldots,\pm2\pi i$, and the singularity at 0. This formula for $\zeta(s)$ can be regarded as an approximate functional equation, where the remainder is expressed explicitly in terms of the contour integral over $\mathcal C_M$. The main task in developing the Riemann-Siegel formula then falls to estimating the contour integral over $\mathcal C_M$ using the saddle-point method.

Prior to Siegel's work, in 1903 Gram showed that the first 10 zeros of $\zeta(s)$ lie on the critical line, and showed that these 10 were the only zeros up to height 50. The development of the above, along with a cogent narrative of Riemann, Siegel, and Gram's contributions, may be found in Edwards \cite{Ed}.

In almost every decade in the last century, mathematicians have set new records for computations of critical zeros of $\zeta(s)$. Alan Turing brought the computer to bear on the problem of computing zeta zeros for the first time in 1950, when, as recounted by Hejhal and Odlyzko~\cite{HejOd}, Turing used the Manchester Mark 1 Electronic Computer, which had 25,600 bits of memory and punched its output on teleprint tape in base 32, to verify every zero up to height 1540 in the critical strip (he found there are 1104 such zeros). Turing also introduced a simplified algorithm to compute zeta zeros now known as Turing's method. Turing published on his computer computations and his new algorithm for the first time in 1953~\cite{Turing}.

Following Turing, the computation of zeros of $\zeta(s)$ took off thanks to the increasing
power of the computer. At this time, the first $10^{13}$ nontrivial zeros of $\zeta(s)$,
tens of billions of nontrivial zeros around the $10^{23}$ and $10^{24}$, and hundreds of
nontrivial zeros near zero number $10^{32}$ are known to lie on the critical line.
Additionally, new algorithms by Sch\"onhage and Odlyzko, and by Sch\"onhage, Heath-Brown,
and Hiary have sped up the verification of zeta zeros.

However, the aforementioned projects for numerically checking that zeros of $\zeta(s)$
lay on the critical line were not concerned with accurately recording the height along the
critical line of the zeros computed; only with ensuring the zeros had real part exactly $1/2$.
This changed in the late 1970s with a series of
computations by Andrew Odlyzko, who was motivated not only by the Riemann Hypothesis but
also by Montgomery's pair correlation conjecture.

Rather than verify consecutive zeros starting from the critical point, Odlyzko was
interested in starting his search high up in the critical strip, in the hope that
near zero number $10^{12}$, the behavior of $\zeta(s)$ would be closer to its asymptotic
behavior. For, as Montgomery's pair correlation conjecture is a statement about the limit
as one's height in the critical strip passes to infinity, one would wish to know the
ordinates of many consecutive zeta zeros in the regime where $\zeta(s)$ is behaving
asymptotically if one wished to test the plausibility of the conjecture.

As he explains~\cite{Od2}, his first computations~\cite{Od1} were in a window around
zero number $10^{12}$, and were done on a Cray supercomputer using the Riemann-Siegel
formula. These computations motivated Odlyzko and Arnold Sch\"onhage to develop a faster
algorithm for computing zeros~\cite{OdAlg,OS}, which was implemented in the late 1980s
and was subsequently used to compute several hundred million zeros near zero number
$10^{20}$ and some near number $2\cdot 10^{20}$, as seen in~\cite{Od4,Od5}.

%%%%%%%%%%%%%%%%%%%%%%%%%%%%%%%%%%%%%%%%%%%%%%%%%%%%%%%%%%%%%%%%%%%%%%%%%%%%%%%%%%%%%%%%%%%%%%%%%%%%%%%%%%%%%%%%%%%%%%%%%%%%%%%%%%%%
%%%%%%%%%%%%%%%%%%%%%%%%%%%%%%%%%%%%%%%%%%%%%%%%%%%%%%%%%%%%%%%%%%%%%%%%%%%%%%%%%%%%%%%%%%%%%%%%%%%%%%%%%%%%%%%%%%%%%%%%%%%%%%%%%%%%
%%%%%%%%%%%%%%%%%%%%%%%%%%%%%%%%%%%%%%%%%%%%%%%%%%%%%%%%%%%%%%%%%%%%%%%%%%%%%%%%%%%%%%%%%%%%%%%%%%%%%%%%%%%%%%%%%%%%%%%%%%%%%%%%%%%%
%%%%%%%%%%%%%%%%%%%%%%%%%%%%%%%%%%%%%%%%%%%%%%%%%%%%%%%%%%%%%%%%%%%%%%%%%%%%%%%%%%%%%%%%%%%%%%%%%%%%%%%%%%%%%%%%%%%%%%%%%%%%%%%%%%%%
\subsection{Number Theory and Random Matrix Theory Successes}\label{sec:earlyyearsNTandRMT}

After its introduction as a conjecture in the late 1950s to describe the energy levels of heavy nuclei, random matrix theory experienced successes on both the numerical and the experimental fronts. The theory was beautifully developed to handle a large number of statistics, and many of these predictions were supported as more and more data on heavy nuclei became available. While there was significant theoretical progress (see, among others, \cite{Dy1, Dy2, Gau, Meh1, MG, Wig1, Wig2, Wig3, Wig4, Wig5, Wig6}), there were some gaps that were not resolved until recently. For example, while  the density of normalized eigenvalues in matrix ensembles (Wigner's semi-circle law) was known for all ensembles where the entries were chosen independently from nice distributions, the spacings between adjacent normalized eigenvalues resisted proof until this century (see, among others, \cite{ERSY, ESY, TV1, TV2}).

The fact that random matrix theory also had a role to play in number theory only emerged roughly twenty years after Wigner's pioneering investigations. The cause of the connection was a chance encounter between Hugh Montgomery and Freeman Dyson at the Institute for Advanced Study at Princeton. As there are now many excellent summaries and readable surveys of their meeting, early years and statistics (see in particular \cite{Ha} for a Hollywoodized version), and the story is now well known, we content ourselves with a quick summary. For more, see among others \cite{Con2, Con3, Di1, Di2, IK, KaSa1, KaSa2, KeSn3, MT-B}.

As described in \S\ref{sec:classnumberproblem}, Montgomery was interested in the class number, which led him to study the pair correlation of zeros of the Riemann zeta function. Given an increasing sequence $\{\alpha_n\}_{n=1}^\infty$ and a box $B \subset \mathbb{R}^{n-1}$, the $n$-level correlation is defined by \begin{eqnarray}\label{eqnlevelcorragain}
\lim_{N \rightarrow \infty} \frac{\#\left\{\left(\alpha_{j_1}-\alpha_{j_2}, \dots, \alpha_{j_{n-1}} - \alpha_{j_n}\right) \in B, j_i \neq j_k  \right\}}{N};
\end{eqnarray} the pair correlation is the case $n=2$, and through combinatorics knowing all the correlations yields the spacing between adjacent events. Montgomery was partially able to determine the behavior for the pair correlation. When he told Dyson his result, Dyson recognized it as the pair correlation function of eigenvalues of random Hermitian matrices in a Gaussian Unitary Ensemble, GUE.

This observation was the beginning of a long and fruitful relationship between the two areas. At first it appeared that the GUE was the only family of matrices needed for number theory, as there was remarkable universality seen in statistics. This ranged from work by Dennis Hejhal \cite{Hej} on the 3-level correlation of the zeros of $\zeta(s)$ and Zeev Rudnick and Peter Sarnak \cite{RS} on the $n$-level correlation of general automorphic $L$-functions, to Odlyzko's \cite{Od1,Od2} striking experiments on spacings between adjacent normalized zeros. In all cases the behavior agreed with that of the GUE.

In particular, Odlyzko's computations of high zeta zeros showed that,
high enough along the critical line, the empirical distribution of nearest-neighbor
spacings for zeros of $\zeta(s)$ becomes more or less indistinguishable from that of
eigenvalues of random matrices from the Gaussian Unitary Ensemble, or GUE.
The agreement with the first million zeros is poor, but the agreement near zero number
$10^{12}$ is close, near perfect near zero number $10^{16}$, and even better near
zero number $10^{20}$. These results provide massive evidence for Montgomery's conjecture,
and vindicate Odlyzko's choice of starting his search high along the critical line; see
Figure~\ref{fig:odlyzko_hump}.

\begin{figure}[h]
\begin{center}
\scalebox{.35}{\includegraphics{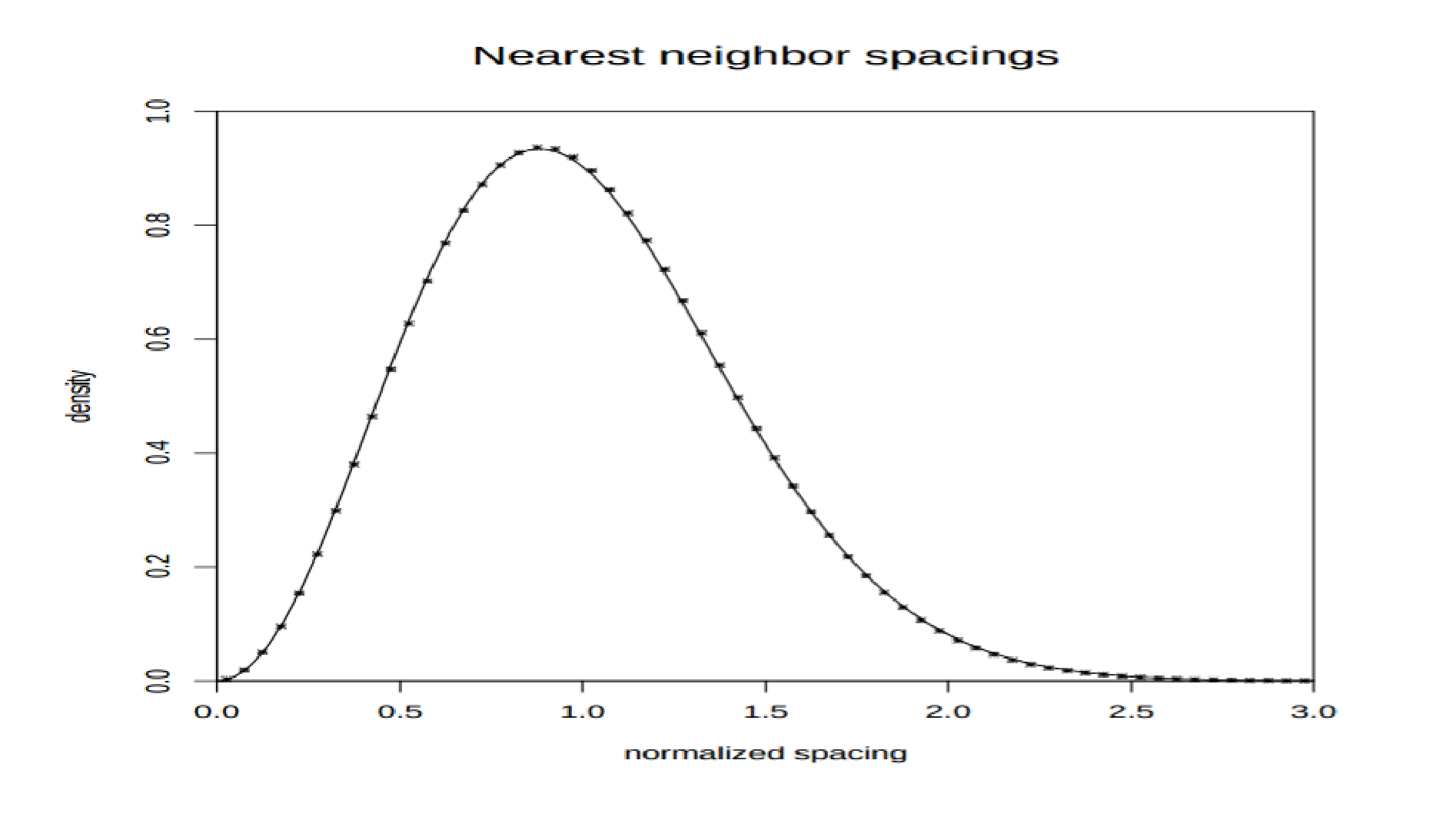}}
\caption{\label{fig:odlyzko_hump}
  Probability density of the normalized spacings $\delta_n$.
  Solid line: GUE prediction. Scatterplot: empirical data based on Odlyzko's computation of
  a billion zeros near zero $\#1.3\times10^{16}$. (From Odlyzko~\cite[Figure 1, p. 4]{Od2}.)}
\end{center}
\end{figure}

In all of these investigations, however, the statistics studied are insensitive to the behavior of finitely many zeros. This is a problem, as certain zeros of $L$-functions play an important role. The most important of these are those of elliptic curve $L$-functions. Numerical computations on the number of points on elliptic curves modulo $p$ led to the Birch and Swinnerton-Dyer conjecture. Briefly, this states that the order of vanishing of the $L$-function at the central point equals the geometric rank of the Mordell-Weil group of rational solutions. The theorems on $n$-level correlations and spacings between adjacent zeros are all limiting statements; we may remove finitely many zeros without changing these limits. Thus these quantities cannot detect what is happening at the central point.

Unfortunately for those who were hoping to distinguish between different symmetry groups,
Nick Katz and Peter Sarnak~\cite{KaSa1, KaSa2} showed in the nineties that the $n$-level correlations of the scaling limits of the classical compact groups are all the same and equal that of the GUE. Thus when we were saying number theory agreed with GUE we could instead have said it agreed with unitary, symplectic or orthogonal matrices.

This led them to develop a new statistic that would be sensitive to finitely many zeros in general, and the important ones near the central point in particular. The resulting quantity is the $n$-level density. We assume the Generalized Riemann Hypothesis (GRH) for ease of exposition, so given an $L(s,f)$ all the zeros are of the form $1/2 + i\gamma_{j;f}$ with $\gamma_{j;f}$ real. The statistics are still well-defined if GRH fails, but we lose the interpretation of ordered zeros and connections with nuclear physics. For more detail on these statistics see the seminal work by Henryk Iwaniec, Wenzhi Luo and Peter Sarnak \cite{ILS}, who introduced them (or \cite{AAILMZ} for an expanded discussion).

Let $\phi_j$ be even Schwartz functions such that the Fourier transforms \begin{equation} \widehat{\phi_j}(y) \ := \ \int_{-\infty}^\infty \phi_j(x) e^{-2\pi i xy} dx\end{equation} are compactly supported, and set $\phi(x) = \prod_{j=1}^n \phi_j(x_j)$. The $n$-level density for $f$ with test function $\phi$ is
\begin{eqnarray}
D_n(f,\phi) & \ = \ &  \sum_{j_1, \dots, j_n \atop j_\ell \neq j_m} \phi_1\left(L_f \gamma_{j_1;f}\right) \cdots \phi_n\left(L_f
\gamma_{j_n;f}\right),
\end{eqnarray} where $L_f$ is a scaling parameter which is frequently related to the conductor. Given a family $\mathcal{F} = \cup_N \mathcal{F}_N$ of $L$-functions with conductors tending to infinity, the $n$-level density $D_n(\mathcal{F},\phi,w)$ with test function $\phi$ and non-negative weight function $w$ is defined by \begin{equation} D_n(\mathcal{F},\phi,w) \ := \ \lim_{N \to \infty} \frac{\sum_{f\in \mathcal{F}_N} w(f) D_n(f,\phi)  }{ \sum_{f\in \mathcal{F}_N} w(f)}. \end{equation}

Katz and Sarnak \cite{KaSa1, KaSa2} conjecture that as the conductors tend to infinity, the $n$-level density of zeros near the central point in families of $L$-functions agree with the scaling limits of eigenvalues near 1 of classical compact groups. Determining \emph{which} classical compact group governs the symmetry is one of the hardest problems in the subject, though in many cases through analogies with a function field analogue one has a natural candidate for the answer, arising from the monodromy group. Unlike the $n$-level correlations, the different classical compact groups all have different scaling limits. As the test functions are Schwartz and of rapid decay, this statistics \emph{is} sensitive to the zeros at the central point. While it was possible to look at just one $L$-function when studying correlations, that is not the case for the $n$-level density. The reason is that while one $L$-function has infinitely many zeros, it only has a finite number within a small, bounded window of the central point (the size of the window is a function of the analytic conductor). We always need do perform some averaging; for the $n$-level correlations each $L$-function gives us enough zeros high up on the critical line for such averaging, while for the $n$-level density we must move horizontally and look at a \emph{family} of $L$-functions. While the exact definition of family is still a work in progress, roughly it is a collection of $L$-functions coming from a common process. Examples include Dirichlet characters, elliptic curves, cuspidal newforms, symmetric powers of ${\rm GL}(2)$ $L$-functions, Maass forms on ${\rm GL}(3)$, and certain families of ${\rm GL}(4)$ and ${\rm GL}(6)$ $L$-functions; see for example \cite{AAILMZ, AM, DM1, DM2, ER-GR, FiM, FI, Gao, Gu, HM, HR, ILS, KaSa2, LM, Mil1, MilPe, OS1, OS2, RR, Ro, Rub, Ya, Yo2}. This correspondence between zeros and eigenvalues allows us, at least conjecturally, to assign a definite symmetry type to each family of $L$-functions (see \cite{DM2, ShTe} for more on identifying the symmetry type of a family).

There are many other quantities that can be studied in families. Instead of looking at zeros, one could look at values of $L$-functions at the central point, or moments along the critical line. There is an extensive literature here of conjectures and results, again with phenomenal agreement between the two areas. See for example \cite{CFKRS}.

%%%%%%%%%%%%%%%%%%%%%%%%%%%%%%%%%%%%%%%%%%%%%%%%%%%%%%%%%%%%%%%%%%%%%%%%%%%%%%%%%%%%%%%%%%%%%%%%%%%%%%%%%%%%%%%%%%%%%%%%%%%%%%%%%%%%
%%%%%%%%%%%%%%%%%%%%%%%%%%%%%%%%%%%%%%%%%%%%%%%%%%%%%%%%%%%%%%%%%%%%%%%%%%%%%%%%%%%%%%%%%%%%%%%%%%%%%%%%%%%%%%%%%%%%%%%%%%%%%%%%%%%%
%%%%%%%%%%%%%%%%%%%%%%%%%%%%%%%%%%%%%%%%%%%%%%%%%%%%%%%%%%%%%%%%%%%%%%%%%%%%%%%%%%%%%%%%%%%%%%%%%%%%%%%%%%%%%%%%%%%%%%%%%%%%%%%%%%%%
%%%%%%%%%%%%%%%%%%%%%%%%%%%%%%%%%%%%%%%%%%%%%%%%%%%%%%%%%%%%%%%%%%%%%%%%%%%%%%%%%%%%%%%%%%%%%%%%%%%%%%%%%%%%%%%%%%%%%%%%%%%%%%%%%%%%
\section{Future Trends and Questions in Number Theory}\label{sec:futuretrendsquestions}
\numberwithin{equation}{section}

The results above are just a small window of the great work that has been done with number theory and random matrix theory. Our goal above is not to write a treatise, but to quickly review the history and some of the main results, setting the stage for some of the problems we think will drive progress in the coming decades. As even that covers too large an area, we have chosen to focus on a few problems with a strong numeric component, where computational number theory is providing the same support and drive to the subject as experimental physics did years before. There are of course many other competing models for $L$-functions. One is the Ratios Conjectures of Conrey, Farmer and Zirnbauer \cite{CFZ1, CFZ2, CS}. Another excellent candidate is Gonek, Hughes and Keating's hybrid model \cite{GHK}, which combines random matrix theory with arithmetic by modeling the $L$-function as a partial Hadamard product over the zeros, which is modeled by random matrix theory, and a partial Euler product, which contains the arithmetic.

In all of the quantities studied, we have agreement (either theoretical or experimental) of the main terms with the main terms of random matrix theory in an appropriate limit. A natural question to ask is how this agreement is reached; in other words, what is the rate of convergence, and what affects this rate? In the interest of space we assume in parts of this section that the reader is familiar with the results and background material from \cite{ILS,RS}, though we describe the results in general enough form to be accessible to a wide audience.

\subsection{Nearest Neighbor Spacings}

We first look at spacing between adjacent zeros, where Odlyzko's work has shown phenomenal agreement for zeros of $\zeta(s)$ and eigenvalues of the GUE ensemble. We plot the \emph{difference} between the empirical and `theoretical,' or `expected' GUE spacings in Figure \ref{fig:odlyzko_difference}. In his paper \cite{Od2}, Odlyzko writes: \emph{Clearly there is structure in this difference graph, and the challenge is to understand where it comes from.}

\begin{figure}[h]
\begin{center}
\scalebox{.35}{\includegraphics{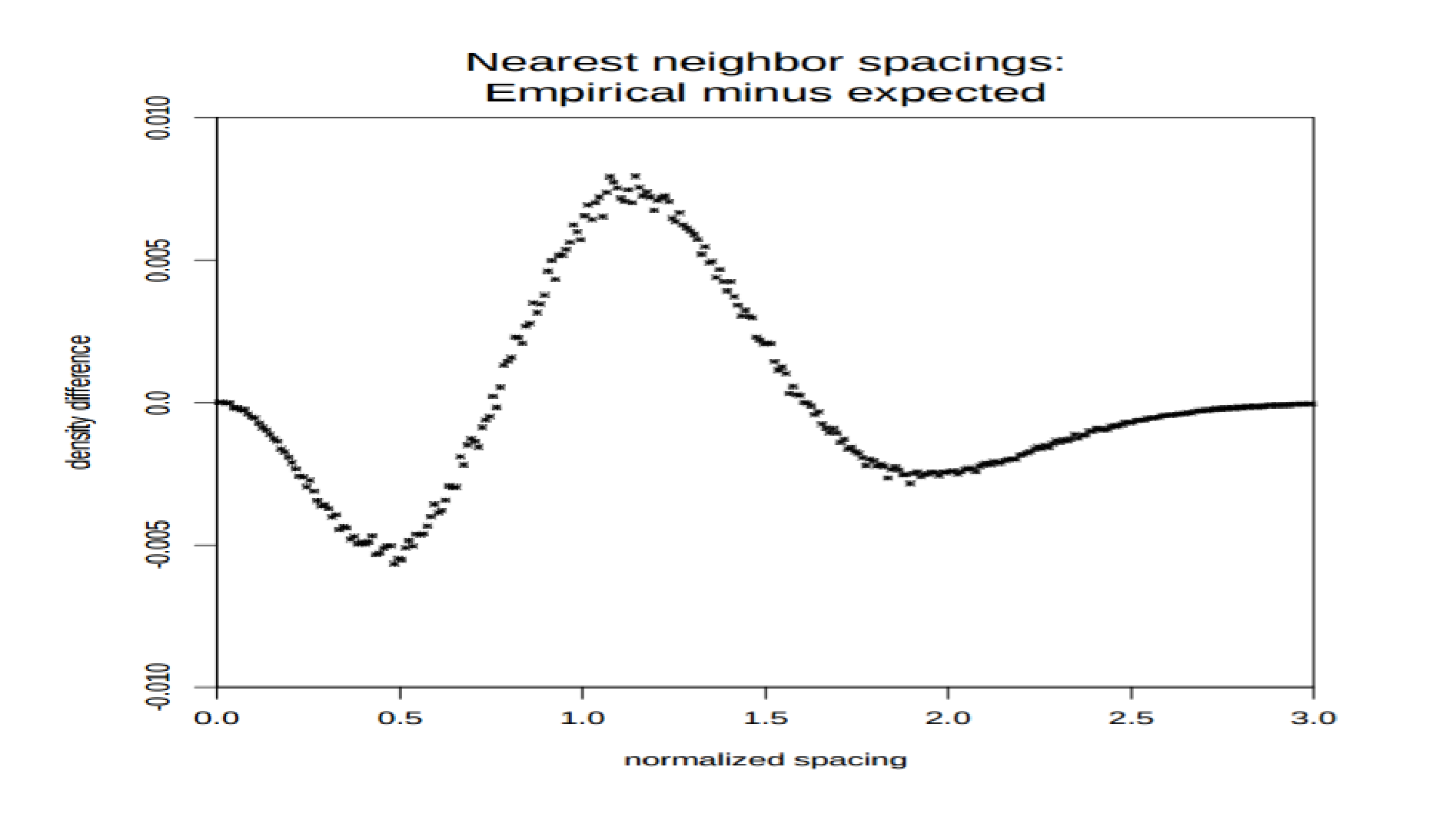}}
\caption{\label{fig:odlyzko_difference}
  Probability density of the normalized spacings $\delta_n$.
  Difference between empirical distribution for a billion zeros near zero
  $\#1.3\times10^{16}$, as computed by Odlyzko, and the GUE prediction.
  (From Odlyzko~\cite[Figure 2, p. 5]{Od2}.)}
\end{center}
\end{figure}

Recently, compelling work of Bogomolny, Bohigas, Leboeuf and Monastra \cite{bogo}
provides a conjectural answer for the source of the additional structure in the form of
lower-order terms in the pair correlation function for $\zeta(s)$. Though the main term
is all that appears in the limit (where Montgomery's conjecture applies), the lower-order
terms contribute to any computation outside the limit, and would therefore influence
any numerical computations like those of Odlyzko. By comparing a conjectural formula for
the two-point correlation function of critical zeros of $\zeta(s)$ of roughly height $T$
due to Bogomolny and Keating in~\cite{BogoKeat} with the known formula for the two-point
correlation function for eigenvalues of unitary matrices of size $N$, Bogomolny et. al. 
deduce a recipe for picking a matrix size that will best model the lower-order terms
in the two-point correlation function, and conjecture that it will be the best choice
for all correlation functions, and therefore the nearest-neighbor spacing. More recently yet, Due\~nez, Huynh, Keating, Miller, and Snaith \cite{DHKMS1, DHKMS2} have applied techniques of Bogomolny et. al. and others to studying lower-order terms in
the behavior of the lowest zeros of $L$-functions attached to elliptic curves. Their results are currently being extended to other $L$-functions by the first and third named authors here and their colleagues.

\subsection{$n$-Level Correlations and Densities}

The results of the studies on spacings between zeros suggest that, while the arithmetic of the $L$-function is not seen in the main term, it does arise in the lower order terms, which determine the \emph{rate}  of convergence to the random matrix theory predictions. Another great situation where this can be seen is through the $n$-level correlations and the work of Rudnick and Sarnak \cite{RS}. They proved that the $n$-level correlations of \emph{all} cuspidal automorphic $L$-functions $L(s,\pi)$ have the same limit (at least in suitably restricted regions). Briefly, the source of the universality in the main term comes from the Satake parameters in the Euler product of the $L$-function, whose moments are the coefficients in the series expansion. In their Remark 3 they write (all references in the quote are to their paper): \begin{quote} The universality (in $\pi$) of the distribution of zeros of $L(s, \pi)$ is somewhat surprising, the reason being that the distribution of the coefficients $a_\pi(p)$ in (1.6), as $p$ runs over primes, is not universal. For example, for degree-two primitive $L$-functions, there are two conjectured possible limiting distributions for the $a_\pi(p)$'s: Sato-Tate or uniform distribution (with a Dirac mass term). As the degree increases, the number of possible limit distributions increases rapidly. However, it is a consequence of the theory of the Rankin-Selberg $L$-functions (developed by Jacquet, Piatetski-Shapiro, and Shalika for $m > 3$) that all these limiting distributions have the same second moment (at least under hypothesis (1.7)). It is the universality of the second moment that is eventually responsible for the universality in Theorems 1.1 and 1.2. For the case of pair correlation ($n = 2$), this is reasonably evident; for $n > 2$ it was (at least for us) unexpected, and it has its roots in a key feature of ``diagonal pairings'' that emerges as the main term in the asymptotics of $R_n(T, f, h)$. \end{quote}
Similar results are seen in the $n$-level densities. There we average the Satake parameters over a family of $L$-function, and in the limit as the conductors tend to infinity only the first and second moments contribute to the main term (at least under the assumption of the Ramanujan conjectures for the sizes of these parameters). The first moment controls the rank at the central point, and the second moment determines the symmetry type (see \cite{DM2,ShTe}). For example, families of elliptic curves with very different arithmetic (complex multiplication or not, or different torsion structures) have the same limiting behavior \emph{but} have different rates of convergence to that limiting behavior. This can be seen in terms of size one over the logarithm of the conductor; while these terms vanish as the conductors tend to infinity, they are present for finite values. See \cite{Mil2, Mil5} for several examples (as well as \cite{MMRW}, where interesting biases are observed in lower order terms of the second moments in the families).

\subsection{Conclusion}

The number theory results above may be interpreted in a framework similar to that of the Central Limit Theorem. There, if we have `nice' independent identically distributed random variables, their normalized sum (standardized to have mean zero and variance 1) converges to the standard normal distribution. The remarkable fact is the universality, and that the limiting distribution is independent of the shape of the distribution. We quickly review why this is the case and then interpret our number theory results in a similar vein.

Given a distribution with finite mean and variance, we can always perform a linear change of variables to study a related quantity where now the mean is zero and the variance one. Thus, the first moment where the \emph{shape} of the distribution is noticeable is the third moment (or the fourth if the distribution is symmetric about the mean). In the proof of the Central Limit Theorem through moment generating functions or characteristic functions, the third and higher moments do not survive in the limit. Thus their effect is only on the rate of convergence to the limiting behavior (see the Berry-Esseen theorem), and not on the convergence itself.

The situation is similar in number theory. The higher moments of the Satake parameters (which control the coefficients of the $L$-functions) again surface only in terms which vanish in the limit, and their effect therefore is seen only in the rate of convergence.

This suggests several natural questions. We conclude with two below, which we feel will play a key role in studies in the years to come. These two questions provide a nice mix, with the first related to the main term and the second related to the rate of convergence.\\ \

\begin{itemize}

\item Is Montgomery's pair correlation true for all boxes (or test functions)? What about the $n$-level correlations, both for $\zeta(s)$ and cuspidal automorphic $L$-functions? Note agreement with random matrix theory for all these statistics implies the conjectures on spacings between adjacent zeros. \\ \

\item For a given $L$-function (if we are studying $n$-level correlations) or a family of $L$-functions (if we are studying $n$-level densities), how does the arithmetic enter? Specifically, what are the possible  lower order terms? How are these affected by properties of the $L$-functions? If we use Rankin-Selberg convolution to create new $L$-functions, how is the arithmetic of the lower order terms here a function of the arithmetic of the constituent pieces?\\ \

\end{itemize}

There are numerous resources and references for those wishing to pursue these questions further. For the $n$-level correlations, the starting point are the papers \cite{Mon, Hej, RS}, while for the $n$-level densities it is \cite{KaSa1, KaSa2, ILS}.

%\section*{Appendix}
%\numberwithin{equation}{section}
%\addcontentsline{toc}{section}{Appendix}
%
%
%When placed at the end of a chapter or contribution (as opposed to at the end of the book), the numbering of tables, figures, and equations in %the appendix section continues on from that in the main text. Hence please \textit{do not} use the \verb|appendix| command when writing an %appendix at the end of your chapter or contribution. If there is only one the appendix is designated ``Appendix'', or ``Appendix 1'', or %``Appendix 2'', etc. if there is more than one.

\ \\

\end{document}